\documentclass{hha}
\usepackage{amssymb,amsmath}
\usepackage[all]{xy}
\usepackage{enumerate}
\usepackage{dsfont}
\SelectTips{cm}{}
\usepackage[pdfencoding=auto]{hyperref}

\renewcommand{\theenumi}{\roman{enumi}}
\numberwithin{equation}{section}

\newcommand{\bG}{\mathbb{G}}

\newcommand{\bQ}{\mathbb{Q}}
\newcommand{\bR}{\mathbb{R}}
\newcommand{\bT}{\mathbb{T}}
\newcommand{\bZ}{\mathbb{Z}}

\newcommand{\cE}{\mathcal{E}}
\newcommand{\cP}{\mathcal{P}}

\newcommand{\cotimes}{\text{\small $\bigotimes$}}

\newcommand{\ep}{\varepsilon}
\newcommand{\Ext}{\mathrm{Ext}}
\newcommand{\im}{\mathsf{Im}}
\newcommand{\Int}{\mathsf{Int}}
\newcommand{\bk}{{\mathbf k}}
\newcommand{\one}{\mathds{1}}

\newcommand{\Sq}{\mathsf{Sq}}
\newcommand{\Ver}{\mathsf{Ver}}
\newcommand{\zpi}{\bZ[\frac{1}{p}]}
\newcommand{\zti}{\bZ[\frac{1}{2}]}
\renewcommand{\ge}{\geqslant}
\renewcommand{\le}{\leqslant}

\author{David Benson}
\email{d.j.benson@abdn.ac.uk}

\address{Institute of Mathematics, University of Aberdeen, Aberdeen
  AB24 3UE, United Kingdom}

\author{Pavel Etingof}

\email{etingof@math.mit.edu}

\address{Department of Mathematics, Massachusetts Institute of
  Technology, Cambridge, MA 02139, USA}

\title{On cohomology in symmetric tensor categories in prime characteristic}



\begin{document}

\begin{abstract} 
We describe graded commutative Gorenstein 
algebras $\cE_n(p)$ over a field of
characteristic $p$, and we conjecture that 
$\Ext^\bullet_{\Ver_{p^{n+1}}}(\one,\one)\cong\cE_{n}(p)$, where $\Ver_{p^{n+1}}$
are the new symmetric
tensor categories recently constructed 
in~\cite{Benson/Etingof:2019a,Benson/Etingof/Ostrik,Coulembier}.
We investigate the combinatorics of these
algebras, and the relationship with Minc's partition function, as well
as possible actions of the Steenrod operations on them.

Evidence for the conjecture includes a large number of computations
for small values of $n$. We also provide some theoretical evidence.
Namely, we use a Koszul construction to 
identify a homogeneous system of parameters in $\cE_n(p)$
with a homogeneous system of parameters in 
$\Ext^\bullet_{\Ver_{p^{n+1}}}(\one,\one)$. These parameters have 
degrees $2^i-1$ if $p=2$ and $2(p^i-1)$ if $p$ is odd, for $1\le i \le n$.
This at least shows that
$\Ext^\bullet_{\Ver_{p^{n+1}}}(\one,\one)$ 
is a finitely generated graded commutative algebra with
the same Krull dimension as $\cE_n(p)$. For $p=2$ we also show that 
 $\Ext^\bullet_{\Ver_{2^{n+1}}}(\one,\one)$ has the expected rank $2^{n(n-1)/2}$ as a module over the subalgebra of parameters. 
\end{abstract}

\maketitle


\section{Introduction}

In our paper~\cite{Benson/Etingof:2019a}, we introduced a nested
sequence of incompressible symmetric tensor abelian categories
in characteristic two. These were very recently generalised to all primes in
our work with Ostrik~\cite{Benson/Etingof/Ostrik} and simultaneously
by Coulembier~\cite{Coulembier}. These categories, $\Ver_{p^n}$ and
$\Ver_{p^n}^+$, seem to be new fundamental
objects deserving further study. 

Here, our aim is to state a conjecture describing the ring structure
of $\Ext^\bullet_{\Ver_{p^{n+1}}}(\one,\one)$.
We have made large numbers of computations using the computer
algebra system {\sc Magma}~\cite{Bosma/Cannon/Playoust:1997a},
and we conjecture that the answer should be the graded commutative
$\bk$-algebra $\cE_n(p)$ introduced below, where $\bk$ is a field of
characteristic $p$. After defining these algebras, we prove the following.

\begin{theorem}
For $n\ge 0$, the algebra $\cE_n(p)$ is a graded commutative finitely generated
Gorenstein $\bk$-algebra of Krull dimension $n$. 
If $p=2$ then it is
an integral domain, while for $p$ odd it has nilpotent elements.
The Poincar\'e series $f(q)=\displaystyle\sum_{d\ge 0}q^d\dim\cE_n(p)_d$ is a
rational function of $q$ satisfying $f(1/q)=(-q)^nf(q)$.
\end{theorem}

There are natural inclusion maps $\cE_{n-1}(p)\to\cE_{n}(p)$, and in each degree the
sequence
\[ \bk=\cE_{0}(p)\to \cE_{1}(p) \to \cdots \to \cE_{n-1}(p) \to \cE_n(p) \to \cdots \]
stabilises at some finite stage. So it makes sense to examine the colimit
$ \cE_\infty(p) = \displaystyle\lim_{\substack{\longrightarrow \\ n}}\cE_n(p). $
The Poincar\'e series of this algebra in the case $p=2$ is
Minc's partition function~\cite{Minc:1959a}. We adapt
Andrews' proof of a Rogers--Ramanujan style formula~\cite{Andrews:1981a} 
for the reciprocal of the generating function for this partition function
so that it gives us the Poincar\'e series for $\cE_n(p)$ for all
$n\ge 0$ and all primes $p$.

\begin{theorem}
The dimension of $\cE_n(p)_d$ is equal to $\displaystyle\sum_{m=1}^n N_p(m,d)$,
and the dimension of $\cE_\infty(p)_d$ is equal to
$\displaystyle\sum_{m=1}^\infty N_p(m,d)$, 
where
$\displaystyle \sum_{m,d=0}^\infty N_p(m,d)t^mq^d \ = \ \frac{1}
{\sum_{i=0}^\infty(-1)^it^i\ell_{i,p}(q)} $
and\vspace{-5mm}
\[ \ell_{i,p}(q)=\begin{cases} \displaystyle
\prod_{j=1}^i\frac{q^{2^j-1}}{1-q^{2^j-1}} &p=2 \\ 
\displaystyle
\prod_{j=1}^i\frac{q^{2p^{j-1}(p-1)-1}+q^{2(p^j-1)}}{1-q^{2(p^j-1)}}
& p\text{ odd.} \end{cases} \]
\end{theorem}

The relationship with the symmetric tensor abelian categories constructed
in~\cite{Benson/Etingof:2019a,Benson/Etingof/Ostrik,Coulembier} is as
follows. Since the subcategory $\Ver_{p^n}^+\subset \Ver_{p^n}$ is a direct summand, 
this inclusion induces an isomorphism 
$  \Ext^\bullet_{\Ver_{p^n}}(\one,\one)\cong\Ext^\bullet_{\Ver^+_{p^n}}(\one,\one) $
and so we only consider $\Ver_{p^n}$.

\begin{conjecture}\label{mainconj}
The graded commutative $\bk$-algebra 
$\Ext^\bullet_{\Ver_{p^{n+1}}}(\one,\one)$ is isomorphic to
$\cE_{n}(p)$. The inclusion $\Ver_{p^n}\subset \Ver_{p^{n+1}}$
induces the inclusion map $\cE_{n-1}(p)\to \cE_{n}(p)$.
\end{conjecture}

We have the following computational evidence for this conjecture. In
all characteristics, this is true for $n\le 1$. In characteristic two,
we have checked both the dimensions and the algebra structure
for $n=2$ in all degrees, for $n=3$ up to degree $40$, and for $n=4$
up to degree $26$. For $p=3$, $n=2$, $3$, and for $p=5$, $n=2$, 
we have checked the dimensions
and algebra structure up to degree $40$. These computations were
carried out using the computer algebra package {\sc Magma}.\medskip

In a symmetric tensor abelian category, the Steenrod operations
act on $\Ext^\bullet(\one,\one)$ and satisfy the Cartan formula
and unstable condition, 
as well as the homogeneous form of the Adem relations
in which it is not assumed that the operation $\Sq^0$ ($p=2$), respectively 
$\cP^0$ ($p$ odd) acts as the identity (see~\cite{May:1970a}; the construction 
there extends to the setting of symmetric tensor categories). We investigate the
possibilities for their action on $\cE_n(p)$. Our conclusions are
cleanest when $p=2$. In that case, we show that the only possible
action of the Steenrod operations on $\cE_n(2)$ compatible with the
inclusions is that all $\Sq^i=0$ except for the mandatory
$\Sq^{|x|}(x)=x^2$. This makes the action much more like that on
the cohomology of a $p$-restricted Lie algebra than like that on the cohomology of a
finite group. In the case $p$ odd, the existence of nilpotent elements
interferes with the arguments, and we can only prove a weaker statement.

In the final sections, we provide some theoretical evidence for 
Conjecture \ref{mainconj}, and some tools that may help prove it. Namely, we first consider the  
Koszul complex of the generating object $V$ of $\Ver_{p^{n+1}}$ 
and compute its cohomology. Then we use the Koszul complex to 
express $\Ext^\bullet_{\Ver_{p^{n+1}}}(\one,X)$ as the cohomology 
of an explicit complex of vector spaces. While we cannot yet 
compute this cohomology in general, this construction explains 
the conjectural shape of the answer and provides upper bounds 
for dimensions of the individual Ext spaces. In particular, it implies 
the existence of the subalgebra of parameters, 
$\bk[y_1,...,y_n]\subset \Ext^\bullet_{\Ver_{p^{n+1}}}(\one,\one)$, 
where $\deg(y_i)$ equals $2^i-1$ if $p=2$ and $2(p^i-1)$ if
$p>2$. We show that $\Ext^\bullet_{\Ver_{p^{n+1}}}(\one,\one)$ is 
module-finite over this subalgebra, and for $p=2$ show that the 
rank of this module is $2^{n(n-1)/2}$, as predicted by Conjecture \ref{mainconj}.

More generally, we at least show the following.

\begin{theorem}
The graded commutative $\bk$-algebra
$\Ext^\bullet_{\Ver_{p^{n+1}}}(\one,\one)$
is finitely generated, with Krull dimension $n$.
Moreover, for any $X\in \Ver_{p^{n+1}}$,  
$\Ext^\bullet_{\Ver_{p^{n+1}}}(\one,X)$ is a finitely generated module 
over this algebra. 
\end{theorem}

This confirms Conjecture 2.18 of~\cite{Etingof/Ostrik:2004a} for the categories $\Ver_{p^{n+1}}$. 

Once the Ext algebra is better understood, this will be the starting
point for applying support theory to the categories $\Ver_{p^{n+1}}$,
along the lines of the theory for finite groups, developed by
Carlson and others~\cite{Carlson:1983a}.
For example, one might hope that $\Ext^\bullet_{\Ver_{p^{n+1}}}(\one,\one)$ stratifies
the stable category of $\Ver_{p^{n+1}}$ as a tensor triangulated
category, in the sense of 
Benson, Iyengar and Krause~\cite{Benson/Iyengar/Krause:2011a}.
This would give a classification of the tensor ideal thick
subcategories, as well as the tensor ideal localising subcategories of
the stable category of the ind-completion.
If Conjecture~\ref{mainconj} holds, then the inclusion of the
subalgebra of parameters $\bk[y_1,\dots,y_n] \hookrightarrow
\Ext^\bullet_{\Ver_{p^{n+1}}}(\one,\one)$ is an inseparable
isogeny. This implies that it induces a bijection on homogeneous prime
ideals, and so $\mathsf{Proj}\, \Ext^\bullet_{\Ver_{p^{n+1}}}(\one,\one)$ is a
weighted projective space.
\medskip

\noindent
{\bf Acknowledgements.} The work of P. E. was supported by the 
NSF grant DMS-1502244. The authors are grateful to V. Ostrik for
useful discussions,
and to Olivier Dudas for communicating Proposition \ref{dudas}. 

\section{Graded algebras}

For a prime $p$ let $\zpi$ denote the ring of integers with $p$ inverted. 
An element of $\zpi$ is a rational number $r=m/n$ 
where $m,n\in \bZ$ and $n$ is a power of $p$. We say that such an 
element $r$ is \emph{even} if $r/2$ is also in $\zpi$ and \emph{odd} otherwise.
So for $p=2$, every element is even. 
If $a\in \zpi$, we write $(-1)^a$ to denote $+1$ if $a$ is even
and $-1$ if $a$ is odd.

We consider 
$\zpi$-graded algebras $R$ over a field $\bk$ of characteristic $p$.
If $x$ is a homogeneous element of $R$, we write $|x|$ for the degree of $x$.
We say that such an algebra is \emph{graded commutative}
if it satisfies $yx = (-1)^{|x||y|}xy$.

If $R$ is a $\zpi$-graded $\bk$-algebra, we write $\Int(R)$ for
the $\bZ$-graded algebra derived from $R$ by means of the 
inclusion of $\bZ$ in $\zpi$. So for $m\in\bZ$, the
homogeneous part of degree $m$ is given by $\Int(R)_m=R_m$.

\begin{example}\label{eg:2*}
Let $\bk$ be a field of characteristic two, and let $\bk[X^{2^*}]$ 
be the algebra generated by the elements $X^{2^n}$ with $n\in\bZ$, 
with the obvious relations $(X^{2^n})^2=X^{2^{n+1}}$. 
This is a $\zti$-graded commutative $\bk$-algebra, 
with $|X^{2^n}|=2^n$. We have $\Int(\bk[X^{2^*}])=\bk[X]$.
\end{example}

\begin{example}
Let $\bk$ be a field of odd characteristic $p$, and let
$\bk[X^{p^*}]\otimes \Lambda(Y)$
be the algebra generated by elements $X^{p^n}$ with $n\in \bZ$ 
and $Y$
with the relations $(X^{p^n})^p=X^{p^{n+1}}$, $Y^2=0$, $XY=YX$.
This is a $\zpi$-graded commutative $\bk$-algebra,
with $|X^{p^n}|=2p^n$ and $|Y|=1$.
We have $\Int(\bk[X^{p^*}]\otimes\Lambda(Y))=\bk[X]\otimes\Lambda(Y)$.
\end{example}

\begin{definition}
We define the \emph{Reynolds operator} $\rho\colon R\to\Int(R)$
to be the map which is the identity on elements of $\Int(R)$ and
zero on homogeneous elements of $R$ whose degree is not 
an integer. 
\end{definition}

\begin{lemma}\label{le:rho}
The map $\rho$ is an $\Int(R)$-module homomorphism.
\end{lemma}
\begin{proof}
Multiplication by elements of $\Int(R)$ 
preserves whether or not the degree of an
element is an integer.
\end{proof}

\begin{proposition}\label{pr:CM}
If $R$ is a Cohen--Macaulay $\bk$-algebra then so is $\Int(R)$.
\end{proposition}
\begin{proof}
For every element of $R$, some power is an element of $\Int(R)$.
So $R$ is an integral extension of $\Int(R)$.
By Lemma~\ref{le:rho}, the Reynolds operator 
$\rho\colon R \to \Int(R)$ is an $\Int(R)$-module homomorphism.
The proposition now follows from 
Proposition~12 of Hochster and Eagon~\cite{Hochster/Eagon:1971a}.
\end{proof}

\section{\texorpdfstring
{The algebra $\cE_n(p)$}
{The algebra ℰ𝑛(𝑝)}}

We treat separately the cases $p=2$ and $p$ odd.

\subsection{\texorpdfstring
{The algebra $\cE_n(2)$}
{The algebra ℰ𝑛(2)}}

In this section, we examine the case $p=2$,
and we let $\bk$ be a field of characteristic two. 

\begin{definition}\label{def:Rn2}
Let $R=R(n,2)$ be the $\zti$-graded
commutative polynomial algebra
$\bk[x_1,\dots,x_n]$ with $|x_i|=\frac{2^i-1}{2^i}$,
and let $\cE_n(2)=\Int(R)$.
\end{definition}

\begin{example} If $n=1$, we have $R=\bk[x_1]$ with $|x_1|=\frac{1}{2}$. The algebra
$\Int(R)$ is generated by $u=x_1^2$, so $\cE_1(2)=\bk[u]$. 
\end{example} 

\begin{example}
If $n=2$, we have $R=\bk[x_1,x_2]$ with $|x_1|=\frac{1}{2}$,
$|x_2|=\frac{3}{4}$. The algebra
$\Int(R)$ is generated by $u=x_1^2$, $v=x_1x_2^2$, $w=x_2^4$.
Then
\[ \cE_2(2)=\Int(R) = {\bf k}[u,v,w]/(uw+v^2) \]
with $|u|=1$, $|v|=2$, $|w|=3$.
\end{example}

\begin{example}
If $n=3$, we have $R={\bf k}[x_1,x_2,x_3]$ with $|x_1|=\frac{1}{2}$,
$|x_2|=\frac{3}{4}$, $|x_3|=\frac{7}{8}$. Then $\cE_3(2)=\Int(R)$ has
a homogeneous system of parameters $y_1=x_1^2$, $y_2=x_2^4$,
$y_3=x_3^8$, of degrees $1,3,7$. The quotient by these parameters
has the following basis. 
\[ \renewcommand{\arraystretch}{1.2}
\begin{array}{|c|ccccccccc|} \hline
\deg & 0 & 1 & 2 & 3 & 4&5&6& 7 & 8 \\ \hline
\text{elt}&1 &&x_1x_2^2& x_1x_2x_3^2 & x_1x_3^4
&x_2^2x_3^4&x_2x_3^6&& x_1x_2^3x_3^6 \\ 
&&&&&x_2^3x_3^2&&&&\\ \hline
\end{array} \]
The Poincar\'e series of $\cE_3(2)$ is therefore given by
\[ \sum_{d\ge 0}q^d\dim\cE_3(2)_d=\frac{1+q^2+q^3+2q^4+q^5+q^6+q^8}{(1-q)(1-q^3)(1-q^7)}. \]
\end{example}

\begin{theorem}\label{th:Gorenstein}
The algebra $\cE_n(2)$ is a Gorenstein integral domain. It has
a regular homogeneous sequence of parameters $y_1=x_1^2,\ y_2=x_2^4,\ y_3=x_3^8,\dots,\ y_n=x_n^{2^n}$
of degrees $1,3,7,\dots,2^n-1$.
Modulo this regular sequence, we get a graded Frobenius algebra
of dimension $2^{\frac{n(n-1)}{2}}$ with dualising element 
$ \alpha=x_1x_2^3x_3^7\dots x_{n-1}^{2^{n-1}-1}x_n^{2^n-2} $
in degree $2^{n+1}-2n-2$. The Poincar\'e series $f(q)=\sum_{d\ge 0}q^d\dim\cE_n(2)_d$ 
is a rational function of $q$ satisfying $f(1/q)=(-q)^nf(q)$.
\end{theorem}
\begin{proof}
It follows from Proposition~\ref{pr:CM} that $\cE_n(2)=\Int({\bf k}[x_1,\dots,x_n])$
is a Cohen--Macaulay integral domain.
So the homogeneous sequence of parameters 
$y_1,y_2,y_3,\dots$, $y_n$ is a regular sequence.

If $x_1^{a_1}\dots x_n^{a_n}$ is a monomial in $\cE_n(2)$ then
$a_n$ is even. If such a monomial is not divisible
by any of the parameters then $a_i\le 2^i-1$ for $1\le i< n$,
and $a_n\le 2^n-2$. The monomial
$ x_1^{1-a_1}x_2^{3-a_2}x_3^{7-a_3}\dots x_{n-1}^{2^{n-1}-1-a_{n-1}} x_n^{2^n-2-a_n}$
is also a basis element of $\cE_n(2)$ and the product of this
with $x_1^{a_1}\dots x_n^{a_n}$ is equal to $\alpha$.
So $\cE_n(2)/(x_1^2,x_2^4,x_3^8,\dots,x_n^{2^n})$ is a 
Frobenius algebra with a basis consisting of these monomials,
and with dualising element $\alpha$.

It is easy to verify using the Frobenius property that $f(1/q)=(-q)^nf(q)$. It then
follows by Theorem~4.4 of Stanley~\cite{Stanley:1978a}
that $\cE_n(2)$ is a Gorenstein algebra. Alternatively, it is shown in
Eisenbud~\cite[\S21.3]{Eisenbud:1995a} that the Gorenstein property
holds for a graded Cohen--Macaulay ring if and only if the
quotient by a regular sequence of parameters is a Frobenius algebra.
\end{proof}

There is a natural inclusion map of algebras $R(n-1,2) \to R(n,2)$
given by sending each $x_i$ in $R(n-1,2)$ to the element with the same
name in $R(n,2)$.
It is easy to check that in each degree the sequence
\[ R(1,2) \to \cdots \to R(n-1,2)\to R(n,2) \to \cdots \]
stabilises at some finite stage. So we take the colimit
$ R(\infty,2) = \displaystyle\lim_{\substack{\longrightarrow \\ n}}R(n,2). $
Applying $\Int$, we obtain inclusion maps
$\cE_{1}(2) \to \cdots \to \cE_{n-1}(2) \to \cE_n(2) \to \cdots$
whose colimit we denote $\cE_\infty(2)=\Int(R(\infty,2))$.

\subsection{\texorpdfstring
{The algebra $\cE_n(p)$, $p>2$}
{The algebra ℰ𝑛(𝑝), 𝑝 > 2}}

For odd primes, we should double the degrees of the polynomial
generators and introduce new exterior generators of degree one
smaller.

Let $p$ be an odd prime and let ${\bf k}$ be a field of characteristic $p$.
Let $R=R(n,p)$ be the $\zpi$-graded commutative algebra
$ {\bf k}[x_1,\dots,x_n] \otimes \Lambda(\xi_1,\dots,\xi_n) $
with $|x_i|=\frac{2(p^i-1)}{p^i}$ and $|\xi_i|=|x_i|-1=\frac{p^i-2}{p^i}$.
Note that $|x_i|$ is even and $|\xi_i|$ is odd.
We define $\cE_n(p)=\Int(R(n,p))$. 

\begin{example} 
If $n=1$, we have $R={\bf k}[x_1]\otimes \Lambda(\xi_1)$ with 
$|x_1|=\frac{2(p-1)}{p}$ and $|\xi_1|=\frac{p-2}{p}$. In this case, the algebra $\cE_1(p)=\Int(R)$ is generated by the elements $y=x_1^p$ and $\eta=x_1^{p-1}\xi_1$
with $|y|=2p-2$, $|\eta|=2p-3$, namely, $\cE_1(p)=\bk[y]\otimes \Lambda(\eta)$. 
\end{example} 

\begin{example}
If $p=3$ and $n=2$, we have $R={\bf k}[x_1,x_2]\otimes \Lambda(\xi_1,\xi_2)$ with 
$|x_1|=\frac{4}{3}$, $|x_2|=\frac{16}{9}$, $|\xi_1|=\frac{1}{3}$,
$|\xi_2|=\frac{7}{9}$. In this case, the algebra $\cE_2(3)=\Int(R)$ is generated by
the following elements: 
\[ \begin{array}[t]{cc}
\textrm{element} & \textrm{degree} \\
x_1^2\xi_1 & 3 \\
x_1^3 & 4 \\
x_1x_2^2\xi_1\xi_2 & 6 \\
x_1x_2^3\xi_1 & 7 
\end{array} \qquad\qquad  
\begin{array}[t]{cc}
\textrm{element} & \textrm{degree} \\  
x_1^2x_2^2\xi_2 & 7 \\
x_1^2x_2^3 & 8 \\
x_2^5\xi_1\xi_2  & 10 \\
x_1x_2^5\xi_2 & 11 
\end{array} \qquad\qquad  
\begin{array}[t]{cc}
\textrm{element} & \textrm{degree} \\  
x_2^6\xi_1 & 11 \\
x_1x_2^6 & 12 \\
x_2^8\xi_2 & 15 \\
x_2^9 & 16 
\end{array} \]
A regular homogeneous system of parameters is given by $y_1=x_1^3$ and $y_2=x_2^9$,
and the quotient by these parameters is a graded Frobenius algebra
with dualising element $x_1^2x_2^8\xi_1\xi_2$ in degree $18$. We have
\[ \sum_{d\ge 0} q^d\dim\cE_2(3)_d = 
\frac{1+q^3+q^6+2q^7+ q^8+q^{10}+2q^{11}+q^{12}+q^{15}+q^{18} }
{(1-q^4)(1-q^{16})}.\]
\end{example}

\begin{theorem}
The ring $\cE_n(p)$ is Gorenstein. It has a homogeneous 
system of parameters $y_1=x_1^p,y_2=x_2^{p^2},\dots,y_n=x_n^{p^n}$ of
degrees $2(p-1),\ 2(p^2-1),\ \dots,2(p^n-1)$. Modulo this regular
sequence, we get a graded Frobenius algebra of dimension $2^np^{\frac{n(n-1)}{2}}$ with dualising
element 
$\alpha=x_1^{p-1}x_2^{p^2-1}\ldots x_n^{p^n-1}\xi_1\xi_2\ldots \xi_n$
in degree $2\bigl(\frac{p^{n+1}-1}{p-1}\bigr)-3n-2$. The Poincar\'e
series $f(q)=\sum_{d\ge 0}q^d\dim\cE_n(p)_d$ is a rational function
of $q$ satisfying $f(1/q)=(-q)^nf(q)$.
\end{theorem}
\begin{proof}
It follows from Proposition~\ref{pr:CM} that $\cE_n(p)$ is
Cohen--Macaulay. Since $y_1,y_2,\dots,$ $y_n$ are
elements of $\cE_n(p)$ which form a regular sequence of parameters in $R(n,p)$,
they also form a regular sequence of parameters in $\cE_n(p)$.  If 
$x_1^{a_1}\dots x_n^{a_n}\xi_1^{\ep_1}\dots \xi_n^{\ep_n}$ 
($\ep_i\in\{0,1\}$ for $1\le i\le n$)
is a monomial in $\cE_n(p)$ which is not divisible by any of the
parameters then $a_i \le p^i-1$ for $1\le i \le n$.  The monomial
\[ x_1^{p-1-a_1}x_2^{p^2-1-a_2}\dots x_n^{p^n-1-a_n}\xi_1^{1-\ep_1}\dots
  \xi_n^{1-\ep_n} \]
is also a basis element of $\cE_n(p)$ and its product with
$x_1^{a_1}\dots x_n^{a_n}\xi_1^{\ep_1}\dots \xi_n^{\ep_n}$ 
is equal to $\alpha$. 

Again it is easy to verify using the Frobenius property that $f(1/q)=(-q)^nf(q)$.
But this time, we cannot show the Gorenstein property as in the proof of
Theorem~\ref{th:Gorenstein}, using Theorem~4.4
of~\cite{Stanley:1978a}, because $\cE_n(p)$ is not an integral domain.
However, the alternative argument using \S21.3
of~\cite{Eisenbud:1995a} still shows that $\cE_n(p)$ is Gorenstein.
\end{proof}

\begin{remark} 
Recall that there is an action of the multiplicative group 
$\bG_m$ on the algebras ${\bf k}[x_1,...,x_n]$ and 
${\bf k}[x_1,...,x_n]\otimes \Lambda (\xi_1,...,\xi_n)$ defined by 
their $\bZ$-grading (the fractional degrees multiplied by $p^n$). 
Also we have the semisimple infinitesimal subgroup scheme 
$\mu_{p^n}\subset \bG_m$ defined by the equation $a^{p^n}=1$ 
(i.e., $\mu_{p^n}=(\bG_m)_{(n)}$, the $n$-th Frobenius kernel 
of $\bG_m$). For $p>2$ we have $\cE_n(p)=({\bf k}[x_1,...,x_n]
\otimes \Lambda (\xi_1,...,\xi_n))^{\mu_{p^n}}$, the subring of
invariants, and for $p=2$ we similarly have 
$\cE_n(2)={\bf k}[x_1,...,x_{n-1},x_n]^{\mu_{2^n}}$ $={\bf k}[x_1,...,x_{n-1},x_n^2]^{\mu_{2^n}}$. 
Since $\sum_i \deg(x_i)-\sum_i \deg(\xi_i)$ is an integer, the action of $\mu_{p^n}$ 
on the super-space spanned by the variables $x_i,\xi_i$ for $p>2$ has Berezinian equal to 1 (recall that degrees of odd variables should be counted with a minus sign). Similarly, for $p=2$ 
the action of $\mu_{2^n}$ on the variables $x_1,...,x_{n-1},x_n^2$ has determinant equal to $1$, as 
$\sum_{i=1}^{n-1}\deg x_i+2\deg x_n$ is an integer. This is related to the fact 
that the ring $\cE_n(p)$ is Gorenstein. For example, for $p=2$ 
this follows from a group scheme generalization of Watanabe's theorem: 
the algebra of invariants for a homogeneous unimodular action of a
finite semisimple group scheme on a polynomial algebra is Gorenstein. 
This is a special case of \cite{Kirkman/Kuzmanovich/Zhang:2009a}, Theorem~0.1. 
\end{remark}

\section{\texorpdfstring
{Generating functions}
{Generating functions}}

\subsection{\texorpdfstring{Generating functions, $p=2$}
{Generating functions, 𝑝 = 2}}

For an integer $d\ge 0$, the degree $d$ part of $\cE_n(2)$ has a basis consisting of the monomials
$x_1^{a_1}x_2^{a_2}\dots x_n^{a_n}$ such that the $a_j$ are non-negative integers, and
\begin{equation*}
\textstyle d=\frac{1}{2}a_1+\frac{3}{4}a_2+\dots + \frac{2^n-1}{2^n}a_n. 
\end{equation*}
The smallest integer degree of a term with $a_j>0$ is $j$, for the monomial $x_1x_2\dots x_{j-1}x_j^2$.
So for $d$ an integer, we must have $a_j=0$ for $j>d$. It follows
that the maps of vector spaces $\cE_1(2)_d\to\cE_2(2)_d\to\cdots$ are
eventually isomorphisms, and 
 $\cE_\infty(2)_d$ is a finite dimensional vector space. It is spanned by the
monomials $x_1^{a_1}x_2^{a_2}\cdots$ with
\begin{equation*}
\textstyle d=\frac{1}{2}a_1+\frac{3}{4}a_2+\frac{7}{8}a_3+\cdots .
\end{equation*}
Such an expression is a \emph{partition} of $d$ into parts $\frac{1}{2},\frac{3}{4},\frac{7}{8},\cdots$.
These are enumerated in sequence A002843 of the On-line Encyclopedia 
of Integer Sequences (which is sequence 405 of  Sloane's Handbook~\cite{Sloane:1973a}).
This sequence has been studied by Minc~\cite{Minc:1959a},
Andrews~\cite{Andrews:1981a}, and 
Flajolet and Prodinger~\cite{Flajolet/Prodinger:1987a}; see also
Nguyen, Schwartz and Tran~\cite{Nguyen/Schwartz/Tran:2009a} for a
context in algebraic topology. The first few terms are 
\begin{gather*}
1, 1, 2, 4, 7, 13, 24, 43, 78, 141, 253, 456, 820, 1472, 2645, 
4749, 8523, 15299, 27456, 49267, 
\\ 
88407, 158630, 284622, 510683, 916271, 1643963, 
2949570, 5292027, 9494758, 
\dots\qquad
\end{gather*}
A few more terms can be found at \url{https://oeis.org/A002843/b002843.txt} .
This sequence grows like $C\lambda^n$, where 
\begin{equation}\label{longnum}
C:=0.74040259366730734...,\quad \lambda:=1.79414718754168546...
\end{equation} 

Our analysis of the generating function $\sum_{d=0}^\infty
q^d\dim\cE_n(2)_d$ follows Andrews~\cite{Andrews:1981a}.
Since there are many misprints in the relevant section of~\cite{Andrews:1981a}, 
and we are doing something slightly different,
we choose to repeat the argument in our context.
The analogous argument for $p$ odd, which we carry out later, 
is not dealt with in~\cite{Andrews:1981a}.

Let $N(m,d)$ be the number of monomials of degree $d$
in $x_1,\dots,x_m$ with $a_m>0$.  Thus the dimension
of $\cE_n(2)_d$ is $\sum_{m=1}^nN(m,d)$, and the dimension of $\cE_\infty(2)_d$ is
$\sum_{m=1}^\infty N(m,d)$. 

We can rewrite these monomials in terms of new variables $z_1,z_2,\dots$ as follows.
Set $z_1=x_1^2$, and $z_i=x_{i-1}^{-1}x_i^2$ for $i\ge 2$. These variables $z_i$
are degree one elements of the larger $\zti$-graded 
ring of Laurent polynomials ${\bf k}[x_1,x_1^{-1},x_2,x_2^{-1},\dots]$. 
Then we have
$ x_1^{a_1}x_2^{a_2}\dots = z_1^{b_1}z_2^{b_2}\dots$
where $a_i=2b_i-b_{i+1}$. The constraints $a_i\ge 0$ translate 
to $2b_i\ge b_{i+1}$ for $i\ge 1$, and since the $b_i$ are eventually zero,
they are all non-negative. Thus $N(m,d)$ is the number of sequences
$(b_1,\dots,b_m)$ of nonnegative integers with $\sum_{i=1}^m b_i=d$, 
and $2b_i\ge b_{i+1}$ for $1\le i< m$. 

Set $\mu_m(q)=\sum_{d=0}^\infty N(m,d)q^d$, and $\mu_0(q)=1$.
We would like to compute $\mu_m(q)$. 

In fact, we will compute a more general generating function, taking into account the degrees with respect to all $z_i$. Introduce auxiliary 
variables $q_1,q_2,...$ corresponding to the statistics 
$b_1,b_2,...$; i.e., we define the multivariate Poincar\'e series of $\cE_n(2)$
$$
\mu_m(q_1,...,q_m):=\sum_{b_1,...,b_m: 2b_i\ge b_{i+1}}^\infty q_1^{b_1}...q_m^{b_m},
$$
so that the usual Poincar\'e series of this algebra is $\mu_m(q)=\mu_m(q,...,q)$. 

Thus we have 
\[ \mu_m(q_1,...,q_m)=\sum_{b_1=1}^\infty \sum_{b_2=1}^{2b_1}\cdots
\sum_{b_m=1}^{2b_{m-1}} q_1^{b_1}...q_m^{b_m}. \]
For the last sum we have
$\displaystyle
\sum_{b_m=1}^{2b_{m-1}}q_m^{b_m} 
= \frac{q_m}{1-q_m}(1-q_m^{2b_{m-1}})
$
and so we obtain
\[ \mu_m= \frac{q_m}{1-q_m}\left(\mu_{m-1}-\sum_{b_1=1}^\infty
\sum_{b_2=1}^{2b_1}\!\!\cdots\!\sum_{b_{m-1}=1}^{2b_{m-2}}q_1^{b_1}\dots
q_{m-2}^{b_{m-2}}(q_{m-1}q_m^2)^{b_{m-1}} \right). \]
Now for the last sum we have
$\displaystyle \sum_{b_{m-1}=1}^{2b_{m-2}}(q_{m-1}q_m^2)^{b_{m-1}}
= \frac{q_{m-1}q_m^2}{1-q_{m-1}q_m^2}(1-(q_{m-1}q_m^2)^{2b_{m-2}} )$
and so we obtain
{\small
\[ 
\mu_m=\frac{q_m}{1-q_m}\Bigl(\mu_{m-1}\!-\!\frac{q_{m-1}q_m^2}{1-q_{m-1}q_m^2}\Bigl(
\mu_{m-2}\!-\!\!\sum_{b_1=1}^\infty\sum_{b_2=1}^{2b_1}
\cdots\!\!\!\!\sum_{b_{m-2}=1}^{2b_{m-3}}\!
q_1^{b_1}\!\dots q_{m-3}^{b_{m-3}}(q_{m-2}q_{m-1}^2q_m^4)^{b_{m-2}}\!\Bigr) \Bigr). 
\]
}%
We continue this way, using induction. At the end, we use $\mu_0=1$.
We obtain
{\small
\[ \sum_{i=1}^m (-1)^i \mu_{m-i}
\Bigl(\frac{q_m}{1-q_m}\Bigr)\Bigl(\frac{q_{m-1}q_m^2}{1-q_{m-1}q_m^2}\Bigr)
\Bigl(\frac{q_{m-2}q_{m-1}^2q_m^4}{1-q_{m-2}q_{m-1}^2q_m^4}\Bigr)
\dots \Bigl(\frac{q_{m-i}...q_{m}^{2^i}}{1-q_{m-i}...q_{m}^{2^i}}\Bigr)
\!=\!\begin{cases}0 & m>0 \\ 1 & m=0. \end{cases} \]
}%
So we set
\[ \displaystyle\ell_m(q_1,...,q_m)= \frac{q_1q_2^3q_3^7\dots q_m^{2^m-1}}
{(1-q_m)(1-q_{m-1}q_m^2)\dots(1-q_1q_2^2...q_m^{2^{m-1}})}, \]
and we have 
$\displaystyle\sum_{i=0}^m(-1)^i\mu_{m-i}\ell_i = \begin{cases} 0 &
  m>0 \\ 1 & m=0.\end{cases} $

Now we introduce another variable $t$, and we have
\[\displaystyle\sum_{m=0}^\infty\sum_{i=0}^m t^{m-i}\mu_{m-i}\cdot (-1)^it^i\ell_i = 1.\] 
Setting $j=m-i$ and 
$$
\mu(t,\bold q):=\sum_{m=0}^\infty\mu_m(q_1,...,q_m)t^m,\ 
\mu(t,q):=\mu(t,q,q,...)=\sum_{m=0}^\infty\mu_m(q)t^m,
$$
we rewrite this as
\begin{equation}\label{eq:Andrews}
\mu(t,\bold q)g(t,\bold q)=1,\ g(t,\bold q):=\sum_{i=0}^\infty(-1)^it^i\ell_i(q_1,...,q_i).
\end{equation}
This yields 
$\mu(t,\bold q)=\frac{1}{g(t,\bold q)}.$
In particular, $\mu(t,q)=\frac{1}{g(t,q)},$
where $g(t,q):=g(t,q,q,...)$.
Thus we obtain the following result. 

\begin{theorem}
We have
$$
\mu(t,\bold q)=\left(\sum_{m=0}^\infty\frac{(-1)^mt^mq_1q_2^3q_3^7\dots q_m^{2^m-1}}{(1-q_m)(1-q_{m-1}q_m^2)\dots(1-q_1q_2^2...q_m^{2^{m-1}})}\right)^{-1}. 
$$
In particular, 
\[ \sum_{m,d=0}^\infty N(m,d)t^mq^d\ =\ 
\left. 1\middle/\sum_{i=0}^\infty \frac{(-1)^it^iq^{1+3+7+\dots+(2^i-1)}}
{(1-q)(1-q^3)(1-q^7)\dots(1-q^{2^i-1})}\right. \]
\end{theorem}
Note that $1+3+7+\dots+(2^i-1)=2^{i+1}-i-2$.

Expanding this out, the reciprocal of the generating function for $N(m,d)$ is
\[ 1 - \frac{tq}{1-q} + \frac{t^2q^4}{(1-q)(1-q^3)} - 
\frac{t^3q^{11}}{(1-q)(1-q^3)(1-q^7)} + \cdots \]
which tabulates as follows:
{\tiny
\[ \setlength{\arraycolsep}{1.5mm}
\begin{array}{l|rrrrrrrrrrrrrrrrrrrrr}
&1&q&q^2&q^3&q^4&q^5&q^6&q^7&q^8&q^9&q^{10}&q^{11}
&q^{12}&q^{13}&q^{14}&q^{15}&q^{16}&q^{17}&q^{18} \\ \hline
1&1\\
t&& -1 & -1 & -1&-1&-1&-1&-1&-1&-1&-1&-1&-1&-1&-1&-1&-1&-1&-1\\
t^2&&&&&1&1&1&2&2&2&3&3&3&4&4&4&5&5&5\\
t^3&&&&&&&&&&&&-1&-1&-1&-2&-2&-2&-3&-4
\end{array} \]
}
Taking the reciprocal, we obtain the table of coefficients $N(m,d)$:
{\tiny
\[ \setlength{\arraycolsep}{1.5mm}
\begin{array}{l|rrrrrrrrrrrrrrrrrrrrr}
&1&q&q^2&q^3&q^4&q^5&q^6&q^7&q^8&q^9&q^{10}&q^{11}
&q^{12}&q^{13}&q^{14}&q^{15}&q^{16}&q^{17}&q^{18} \\ \hline
1&1\\
t   &&1&1&1&1&1&1&1&1&1&1&1&1&1&1&1&1&1&1\\
t^2&&&1&2&2&3&4&4&5&6&6&7&8&8&9&10&10&11&12\\
t^3&&&&1&3&4&6&9&11&14&18&22&26&31&36&41&47&53&60\\
t^4&&&&&1&4&7&11&18&25&33&45&59&74&94&116&139&168&199\\
t^5&&&&&&1&5&11&19&33&51&72&102&141&187&246&319&403&504\\
t^6&&&&&&&1&6&16&31&57&96&146&216&313&436&595&802&1056\\
t^7&&&&&&&&1&7&22&48&94&170&278&432&654&954&1353&1888\\
t^8&&&&&&&&&1&8&29&71&149&287&502&822&1299&1979&2918\\
t^9&&&&&&&&&&1&9&37&101&228&466&867&1497&2470&3922\\
t^{10}&&&&&&&&&&&1&10&46&139&338&732&1442&2623&4520\\
t^{11}&&&&&&&&&&&&1&11&56&186&487&1117&2322&4442\\
t^{12}&&&&&&&&&&&&&1&12&67&243&684&1661&3635\\
t^{13}&&&&&&&&&&&&&&1&13&79&311&939&2413\\
t^{14}&&&&&&&&&&&&&&&1&14&92&391&1263\\
t^{15}&&&&&&&&&&&&&&&&1&15&106&484\\
t^{16}&&&&&&&&&&&&&&&&&1&16&121\\
t^{17}&&&&&&&&&&&&&&&&&&1&17\\
t^{18}&&&&&&&&&&&&&&&&&&&1
\end{array} \]
}

\noindent
The coefficients of the Poincar\'e series for $\cE_n(2)$ are 
given by adding the first $n$ rows of this table, while the 
coefficients of the Poincar\'e series for $\cE_\infty(2)$ are
given by adding all the rows; in other words by setting $t=1$.
Thus, setting $N(d):=\sum_{m\ge 0}N(m,d)$, we get 
\[ \sum_{d=0}^\infty N(d)q^d = \frac{1}{\phi(q)}, \quad 
\phi(q):=\sum_{i=0}^\infty \frac{(-1)^iq^{1+3+7+\dots+(2^i-1)}}
{(1-q)(1-q^3)(1-q^7)\dots(1-q^{2^i-1})}. \]
Note that the series $\phi(q)$ defines an analytic 
function in the disk $|q|<1$, and 
that the numbers $C,\lambda$ in \eqref{longnum} 
are determined as follows: $\lambda=\frac{1}{\alpha}$, 
where $\alpha$ is the smallest positive zero of $\phi(q)$, while 
$C=-\frac{1}{\alpha \phi'(\alpha)}$. 

It is easy to see from this computation that the reciprocal of the
generating function is much easier to compute than the generating
function itself, and has much smaller coefficients. The same will be
true for $p$ odd.

\begin{remark}
Recall (\cite{Benson/Etingof:2019a}) that the category
$\Ver_{2^{n+1}}^+$ is the category of modules in $\Ver_{2^{n}}$ 
over the algebra $A:=\Lambda V$, where $V=X_{n-1}$ is the 
generating object of $\Ver_{2^{n}}$. Thus the group $\bG_m$ 
acts on $A$ by scaling $V$. This action gives rise to an action of 
$\bG_m$ on $\Ext^\bullet_{\Ver_{2^{n+1}}}(\one,\one)$, i.e., a 
$\bZ$-grading on each cohomology group. We expect that on 
$\cE_n(2)$, this grading is given by the degree with respect to the 
variable $z_n$. In other words, we expect that the 2-variable
Poincar\'e series of $\cE_n(2)$ taking into account this grading is
$\mu_m(q,...,q,qv)$.  

So let us compute the generating function 
$\mu(t,q,v):=\sum_{m=0}^\infty \mu_m(q,...,q,qv)t^m$.
Arguing as above, we get 
$\mu(t,q,v)-\mu(t,q)+\mu(t,q)g(t,q,v)=1$,
where 
$$
g(t,q,v):=\sum_{i=0}^\infty \frac{(-1)^it^iq^{2^{i+1}-i-2}v^{2^i-1}}
{(1-qv)(1-q^3v^2)(1-q^7v^4)\dots(1-q^{2^i-1}v^{2^{i-1}})}.
$$
Thus, we have 
$$
\mu(t,q,v)=1+\frac{1-g(t,q,v)}{g(t,q)}.
$$
\end{remark}

\subsection{\texorpdfstring
{Generating functions, $p>2$}
{Generating functions, 𝑝 > 2}}

The details for $p$ odd are similar to those for $p=2$, 
but are quite a bit harder to keep straight. 
So we have chosen to write out the computation again in full.

For an integer $d\ge 0$, 
the degree $d$ part of $\cE_p(n)$ has a basis consisting of the monomials 
$x_1^{a_1}\dots x_n^{a_n}\xi_1^{\ep_1}\dots \xi_n^{\ep_n}$ 
such that the $a_j$ are non-negative integers, each $\ep_j$ is zero or one, and 
\[ \textstyle d=\frac{2p-2}{p}a_1+\frac{2p^2-2}{p^2}a_2+\dots+\frac{2p^n-2}{p^n}a_n
\ +\ \frac{p-2}{p}\ep_1+\frac{p^2-2}{p^2}\ep_2+\dots+\frac{p^n-2}{p^n}\ep_n. \]
Let $N_p(m,d)$ be the number of such monomials in degree $d$ with
$a_m+\ep_m>0$. Thus the dimension of $\cE_n(p)_d$ is $\sum_{m=1}^nN_p(m,d)$, and
the dimension of $\cE_\infty(p)_d$ is $\sum_{m=1}^\infty N_p(m,d)$.

Set $z_1=x_1^p$, $\zeta_1=x_1^{p-1}\xi_1$, and 
$z_i=x_{i-1}^{-1}x_i^p$, $\zeta_i=x_{i-1}^{-1}x_i^{p-1}\xi_i$ for $i\ge 2$.
Then we have 
$ |z_i|=2p-2$, $|\zeta_i|=2p-3$ $(1\le i\le n) $
and
\[ (x_1^{a_1}x_2^{a_2}\dots)(\xi_1^{\ep_1}\xi_2^{\ep_2}\dots) = 
(z_1^{b_1}z_2^{b_2}\dots)(\zeta_1^{\ep_1}\zeta_2^{\ep_2}\dots) \]
where 
$ a_i=pb_i+(p-1)\ep_i-b_{i+1}-\ep_{i+1}$. 
Then the conditions on the $b_i$ and the $\ep_i$ are that 
$b_i$ are non-negative integers, $\ep_i=0$ or $1$, and 
$ pb_i+(p-1)\ep_i\ge b_{i+1}+\ep_{i+1}$ for $i\ge 1$. 

Set $\mu_m(q)=\sum_{d=0}^\infty N_p(m,d)q^d$, and $\mu_0(q)=1$. Then we have
\[ \mu_m(q)=\sum_{b_1+\ep_1=1}^\infty \sum_{b_2+\ep_2=1}^{pb_1+(p-1)\ep_1}\cdots 
\sum_{b_m+\ep_m=1}^{pb_{m-1}+(p-1)\ep_{m-1}} q^{(2p-2)(b_1+\dots+b_m)+(2p-3)(\ep_1+\dots+\ep_m)}. \]
We would like to compute $\mu_m(q)$. As in the case $p=2$, 
we introduce auxiliary 
variables $q_1,q_2,..., w_1,w_2,...$ corresponding to the statistics 
$b_1,b_2,...$, $\varepsilon_1,\varepsilon_2,...$; i.e., we define the multivariate Poincar\'e series of $\cE_n(p)$
\[ \mu_m(q_1,...,q_m; w_1,...,w_m):=\!\!
\sum_{b_1+\ep_1=1}^\infty\!\! \sum_{b_2+\ep_2=1}^{pb_1+(p-1)\ep_1}\cdots\!\!\!
\sum_{b_m+\ep_m=1}^{pb_{m-1}+(p-1)\ep_{m-1}}\!\! 
q_1^{b_1}\dots q_m^{b_m}w_1^{\ep_1}\dots w_m^{\ep_m} \]
so that the usual Poincar\'e series of this algebra is 
\[ \mu_m(q)=\mu_m(q^{2p-2},...,q^{2p-2};q^{2p-3},...,q^{2p-3}). \]  
We have 
\begin{equation}\label{sumfor} 
\sum_{b+\ep=1}^s 
q^{b}w^{\ep}
=\frac{(w+q)(1-q^{s})}{1-q}. \\
\end{equation}
So, summing over $b_m,\ep_m$, we get
\begin{align*}
\mu_m&=\frac{w_m+q_m}{1-q_m}
\biggl(\mu_{m-1}-
\sum_{b_1+\ep_1=1}^\infty\sum_{b_2+\ep_2=1}^{pb_1+(p-1)\ep_1}\dots \\
&\qquad\dots\sum_{b_{m-1}+\ep_{m-1}=1}^{pb_{m-2}+(p-1)\ep_{m-2}}
q_1^{b_1}\dots q_{m-2}^{b_{m-2}}w_1^{\ep_1}\dots w_{m-2}^{\ep_{m-2}}
(q_{m-1}q_m^p)^{b_{m-1}}(w_{m-1}q_m^{p-1})^{\ep_{m-1}} \biggr). 
\end{align*}
Thus, summing over $b_{m-1},\ep_{m-1}$ and using \eqref{sumfor} again,
we have 
{\small
\begin{align*}
\mu_m&=\frac{w_m+q_m}{1-q_m}\biggl(\mu_{m-1}
-\frac{w_{m-1}q_m^{p-1}+q_{m-1}q_m^{p}}{1-q_{m-1}q_m^p}
\biggl(\mu_{m-2}-{}
\sum_{b_1+\ep_1=1}^\infty\sum_{b_2+\ep_2=1}^{pb_1+(p-1)\ep_1}\dots\\
\dots&\!\!\sum_{b_{m-2}+\ep_{m-2}=1}^{pb_{m-3}+(p-1)\ep_{m-3}}
q_1^{b_1}\dots q_{m-3}^{b_{m-3}}w_1^{\ep_1}\dots w_{m-3}^{\ep_{m-3}}
(q_{m-2}q_{m-1}^pq_m^{p^2})^{b_{m-2}}(w_{m-2}q_{m-1}^{p-1}q_m^{p^2-p})^{\ep_{m-2}}
\biggr)\biggr). 
\end{align*}
}
Continuing inductively and using that $\mu_0=1$, we obtain
\[ \sum_{i=0}^m(-1)^i\mu_{m-i}\ell_{i,p} = 
\begin{cases} 0&m>0\\1&m=0 \end{cases} \]
where
\begin{multline*} 
\ell_{i,p}(q)=
\Bigl(\frac{w_m+q_m}{1-q_m}\Bigr)\Bigl(
\frac{w_{m-1}q_m^{p-1}+q_{m-1}q_m^{p}}{1-q_{m-1}q_m^p}\Bigr)
\Bigl(\frac{w_{m-2}q_{m-1}^{p-1}q_m^{p^2-p}+q_{m-2}q_{m-1}^{p}q_m^{p^2}}{1-q_{m-2}q_{m-1}^pq_m^{p^2}}\Bigr)\\
\cdots
\Bigl(\frac{w_{m-i+1}q_{m-i+2}^{p-1}...q_m^{p^{i-1}-p^{i-2}}+q_{m-i+1}q_{m-i+2}^{p}q_m^{p^{i-1}}}{1-q_{m-i+1}q_{m-i+2}^p...q_m^{p^{i-1}}}\Bigr)
\end{multline*}
Introducing a new variable $t$, we rewrite this as
$\displaystyle \Bigl(\sum_{j=0}^\infty t^j\mu_j\Bigr)\Bigl(\sum_{i=0}^\infty (-1)^it^i\ell_{i,p}\Bigr)=1$, 
so $\mu_j$ can be determined from the generating function 
\[ \sum_{j=0}^\infty t^j\mu_j=\frac{1}{\sum_{i=0}^\infty (-1)^it^i\ell_{i,p}}. \]

In particular, setting $w_i=q^{2p-3}$, $q_i=q^{2p-2}$, we get 
$$
\ell_{i,p}(q)=
q^{(2p-2)(p^i-1)-i}\frac{(1+q)(1+q^{2p-1})...(1+q^{2p^{i-1}-1})}{(1-q^{2p-2})(1-q^{2p^2-2})...(1-q^{2p^i-2})}.
$$

Thus we obtain the following result. 

\begin{theorem}
We have
{\small
\[ \sum_{m,d=0}^\infty N_p(m,d)t^m \ =\  
\left(\sum_{i=0}^\infty(-1)^it^iq^{(2p-2)(p^i-1)-i}\frac{(1+q)(1+q^{2p-1})...(1+q^{2p^{i-1}-1})}{(1-q^{2p-2})(1-q^{2p^2-2})...(1-q^{2p^i-2})}\right)^{-1}. \]
}
\end{theorem}

\begin{remark}
Recall (\cite{Benson/Etingof/Ostrik}, Subsection 4.14) that the
principal block of the category $\Ver_{p^{n+1}}^+$ (i.e., the block 
of the unit object) is equivalent to the category of modules in 
$\Ver_{p^{n}}$ over the algebra $A:=\Lambda V$, where 
$V=\bT_1$ is the generating object of $\Ver_{p^{n}}$. Thus 
the group $\bG_m$ acts on $A$ by scaling $V$. This action gives 
rise to an action of $\bG_m$ on
$\Ext^\bullet_{\Ver_{p^{n+1}}}(\one,\one)$, i.e., a $\bZ$-grading on 
each cohomology group. We expect that on $\cE_n(p)$, this grading 
is given by the degree with respect to the variables $z_n$ and 
$\zeta_n$. In other words, we expect that the 2-variable Poincar\'e 
series of $\cE_n(p)$ taking into account this grading is 
$\mu_m(q^{2p-2},...,q^{2p-2},(qv)^{2p-2};q^{2p-3},...,q^{2p-3},(qv)^{2p-3})$. 

So let us compute the generating function 
$$
\mu(t,q,v):=\sum_{m=0}^\infty \mu_m(q^{2p-2},...,q^{2p-2},(qv)^{2p-2};q^{2p-3},...,q^{2p-3},(qv)^{2p-3})t^m.
$$
Arguing as above, we get 
$\mu(t,q,v)-\mu(t,q)+\mu(t,q)g(t,q,v)=1$,
where 
{\small
\begin{gather*} 
g(t,q,v):=\\
 \sum_{i=0}^\infty  
\frac{(-1)^it^iq^{(2p-2)(p^i-1)-i}v^{(2p-2)(p^{i-1}-1)-1}(1+qv)(1+q^{2p-1}v^{2p-2})...(1+q^{2p^{i-1}-1}v^{(2p-2)p^{i-2}})}
{(1-q^{2p-2}v^{2p-2})(1-q^{2p^2-2}v^{(2p-2)p})...(1-q^{2p^i-2}v^{(2p-2)p^{i-1}})}.
\end{gather*}
}%
Thus, we have 
$$
\mu(t,q,v)=1+\frac{1-g(t,q,v)}{g(t,q)},
$$
where $g(t,q):=g(t,q,1)$. 
\end{remark}

Here is a table of the coefficients in the reciprocal of the generating function for 
$N_p(m,d)$ with $p=3$.
{\tiny
\[ \setlength{\arraycolsep}{0.22mm}\hspace{-3mm}
\begin{array}{l|rrrrrrrrrrrrrrrrrrrrrrrrrrrrrrrrrrrrrrr}
&1&q\,&q^2&q^3&q^4&q^5&q^6&q^7&q^8&q^9&q^{10}&q^{11}
&q^{12}&q^{13}&q^{14}&q^{15}&q^{16}&q^{17}&q^{18}&q^{19\!}&
q^{20\!}&q^{21\!}&q^{22\!}&q^{23\!}&q^{24\!}&q^{25\!}&q^{26\!}&q^{27\!}&q^{28\!}
&q^{29\!}&q^{30\!}&q^{31\!}&q^{32\!}&q^{33\!}&q^{34\!}&q^{35\!} &q^{36\!}
 \\ \hline
1&1\\
t&&&&-1&-1&&&-1&-1&&&-1&-1&&&-1&-1&&&-1&-1&&&-1&-1
&&&-1&-1&&&-1&-1&&&-1&-1\\
t^2&&&&&&&&&&&&&&&1&1&&&1&2&1&&1&2&1&&1&2&1
&&2&3&1&&2&4&2\\
t^3&&&&&&&&&&&&&&&&&&&&&&&&&&&&&&&&&&&&-1&-1
\end{array} \]
}
Reciprocating, we obtain the table of coefficients $N_3(m,d)$. These
tables become sparser as the prime increases.

\section{\texorpdfstring
{Action of the Steenrod operations}
{Action of the Steenrod operations}}

In this section, we examine possible actions of the Steenrod
operations on the algebra $\cE_\infty(p)$. 

\subsection{\texorpdfstring
{Steenrod operations for $p=2$} 
{Steenrod operations for 𝑝 = 2}}

We begin with the easier
case $p=2$.

\begin{theorem}
There
is only one possibility for the action of the Steenrod operations
on $\cE_\infty(2)$ in such a way that the Cartan formula 
\[ \Sq^n(xy) = \sum_{i+j=n}\Sq^i(x)\Sq^j(y) \]
and the unstable conditions $\Sq^i(x) = x^2$ for 
$i=|x|$ and $\Sq^i(x)=0$ for $i>|x|$ hold.
Namely for $x\in\cE_\infty(2)$,  we have $\Sq^{|x|}(x)=x^2$, and
$\Sq^i(x)=0$ for  $i \ne |x|$.
In particular, if $x$ has degree greater than zero then $\Sq^0(x)=0$.
\end{theorem}
\begin{proof}
We begin by examining the elements $x_n^{2^n}$ of degree $2^n-1$,
and we show by induction on $n$ that $\Sq^i(x_n^{2^n})=0$ for
$i<2^n-1$. 
Let $T=\Sq^0+\Sq^1+\Sq^2+\cdots$ be the total Steenrod
operation, which by the Cartan formula is a ring homomorphism.
In particular, note that
$\Sq^i$ of a $2^n$th power
vanishes when $i$ is not divisible by $2^n$.
Our goal is to show that $T(x_n^{2^n})=(x_n^{2^n})^2$ for all $n\ge 1$.

We begin with $n=1$. We have
$ (x_1^2)(x_2^4) = (x_1x_2^2)^2. $
Applying $\Sq^3$ to this relation,
we obtain 
\[ \Sq^0(x_1^2)(x_2^4)^2+ (x_1^2)^2 \Sq^2(x_2^4) =
  \Sq^3((x_1x_2^2)^2) = 0. \]
Therefore $\Sq^0(x_1^2)$ is divisible by $(x_1^2)^2$, and
is hence zero, and so $T(x_1^2)=(x_1^2)^2$.

Now for the inductive step. Assume that 
$T(x_{n-1}^{2^{n-1}})=(x_{n-1}^{2^{n-1}})^2$.
We have the relation 
\[ (x_{n-1}^{2^{n-1}})^{2^{n-1}-1}\,(x_n^{2^n})=
(x_{n-1}^{2^{n-1}-1}x_n^2)^{2^{n-1}}. \]
in $\cE_\infty(2)$.
Applying $T$, we get
\[ (x_{n-1}^{2^{n-1}})^{2^n-2}\,T(x_n^{2^n})=
(T(x_{n-1}^{2^{n-1}-1}x_n^2))^{2^{n-1}}. \]
The right hand side is zero in degrees not divisible by $2^{n-1}$.
It follows that $T(x_n^{2^n})$ is zero in degrees not congruent
to minus two modulo $2^{n-1}$. So the only possibilities for non-zero
Steenrod operations on $x_n^{2^n}$ are $\Sq^{2^n-1}$ and
$\Sq^{2^{n-1}-1}$.

We also have the relation 
\[ (x_n^{2^n})(x_{n+1}^{2^{n+1}})^{2^n-1}=
(x_nx_{n+1}^{2^{n+1}-2})^{2^n} \]
in $\cE_\infty(2)$. Applying $T$, we get
\[ T(x_n^{2^n})(T(x_{n+1}^{2^{n+1}}))^{2^n-1} = 
(T(x_nx_{n+1}^{2^{n+1}-2}))^{2^n}. \]
The right hand side is zero in degrees not divisible by $2^n$.
So in particular, examining the term in degree $2^{n+1}(2^n-1)-2^{n-1}$, we have
\[ (x_n^{2^n})^2\Sq^{(2^{n+1}-1)(2^n-1)-2^{n-1}}((x_{n+1}^{2^{n+1}})^{2^n-1})+
(\Sq^{2^{n-1}-1}(x_n^{2^n}))(x_{n+1}^{2^{n+1}})^{2^{n+1}-2} = 0. \]
So $\Sq^{2^{n-1}-1}(x_n^{2^n})$ is divisible by $(x_n^{2^n})^2$,
and is hence zero. Hence $T(x_n^{2^n})=(x_n^{2^n})^2$, and
the inductive step is complete.

Finally, given any monomial $x=x_1^{a_1}\dots x_n^{a_n}
\in\cE_\infty(2)$, we raise it to the $2^n$th power to obtain
an element of the subring generated by $x_1^2, x_2^4, x_3^8,\cdots$.
Then 
$ T(x)^{2^n}=T(x^{2^n})=(x^{2^n})^2=(x^2)^{2^n}$, 
and  since we are in an integral domain of characteristic two,
this implies that $T(x)=x^2$.
\end{proof}

\subsection{\texorpdfstring
{Steenrod operations for $p>2$}
{Steenrod operations for 𝑝 > 2}}

Next, we examine possible actions of the Steenrod
operations on the algebra $\cE_\infty(p)$ for $p$ odd. Our conclusions
are weaker than in the case $p=2$, because of the existence of nilpotent elements.

\begin{theorem}
Suppose that the Steenrod operations act on $\cE_\infty(p)$ with $p$
odd in such a way that the Cartan formula and unstable conditions
hold. Then on the subring spanned by the monomials not involving any
of the $\xi_i$, we have $\cP^m(x)=x^p$ and $\cP^i(x)=0$ for $i\ne n$,
where $|x|=2m$.
\end{theorem}
\begin{proof}
Let $T$ be the total Steenrod operation $\cP^0+\cP^1+\cdots$.
The argument to show that $T(x_n^{p^n})=(x_n^{p^n})^p$ for $p$ odd is similar
to the case $p=2$, but involves one more induction. We therefore write
it out in full.

Our first task is to show that $T(x_1^p)=(x_1^p)^p$.
We begin as before with
\[ (x_1^p)(x_2^{p^2})^{p-1} = (x_1x_2^{p(p-1)})^p, \]
a relation of degree $2p^2(p-1)$.
Applying $\cP^{p^2(p-1)-1}$ to this, we get
\[ \cP^{p-2}(x_1^p)\,(x_2^{p^2})^{p(p-1)}+ (x_1^p)^p\,\cP^{(p^2-1)(p-1)-1}((x_2^{p^2})^{p-1})=0. \]
Therefore $\cP^{p-2}(x_1^p)$ is divisible by $(x_1^p)^p$, and
hence it is zero. We work downwards in degree by induction.
Suppose we have shown that $\cP^{p-i}(x_1^p),\dots,\cP^{p-2}(x_1^p)$
are all zero. Then applying $\cP^{p^2(p-1)-i}$ to the above relation,
we get
\[ \cP^{p-i-1}(x_1^p)\,(x_2^{p^2})^{p(p-1)}+ (x_1^p)^p\,\cP^{(p^2-1)(p-1)-i-1}((x_2^{p^2})^{p-1})=0. \]
Therefore $\cP^{p-i-1}(x_1^p)$ is divisible by $(x_1^p)^p$, and hence it
is zero. Once we reach $i=p-1$, we have completed the proof that
$T(x_1^p)=(x_1^p)^p$.

Next, we suppose that we have already shown that
$T(x_{n-1}^{p^{n-1}})=(x_{n-1}^{p^{n-1}})^p$. We have
the relation
\[ (x_{n-1}^{p^{n-1}})^{p^{n-1}-1}(x_n^{p^n})=
(x_{n-1}^{p^{n-1}-1}x_n^p)^{p^{n-1}} \]
in $\cE_\infty(p)$. Applying $T$, we get
\[ (x_{n-1}^{p^{n-1}})^{p^n-p}\,T(x_n^{p^n})=
(T(x_{n-1}^{p^{n-1}-1}x_n^p))^{p^{n-1}}. \]
The right hand side is zero in degrees not divisible by $p^{n-1}$.
So the only possibilities for non-zero Steenrod operations
on $x_n^{p^n}$ are $\cP^{p^n-ip^{n-1}-1}$ for $0\le i\le p-1$.

We also have the relation 
\[ (x_n^{p^n})(x_{n+1}^{p^{n+1}})^{p^n-1}=
(x_nx_{n+1}^{p^{n+1}-p})^{p^n} \]
in $\cE_\infty(p)$. Applying $T$, we get
\[ T(x_n^{p^n})(T(x_{n+1}^{p^{n+1}}))^{p^n-1}=
(T(x_nx_{n+1}^{p^{n+1}-p}))^{p^n}. \]
The right hand side is zero in degrees not divisible by $p^n$. 
We show by induction on $i$ that $\cP^{p^n-ip^{n-1}-1}(x_n^{p^n})=0$
for $1\le i \le p-1$. If we have proved this for smaller values of
$i$, then we get
\[ \cP^{p^n-ip^{n-1}-1}(x_n^{p^n})\,
(x_{n+1}^{p^{n+1}})^{p^{n+1}-p} +
(x_n^{p^n})^p\,
\cP^{(p^{n+1}-1)(p^n-1)-ip^{n-1}}((x_{n+1}^{p^{n+1}})^{p^n-1}) = 0. \]
So $\cP^{p^n-ip^{n-1}-1}(x_n^{p^n})$ is divisible by $(x_n^{p^n})^p$,
and is hence zero. This completes the proof that
$T(x_n^{p^n})=(x_n^{p^n})^p$. 
\end{proof}

\section{The Koszul complex} 

We assume that $p^n>3$. We will consider the symmetric tensor categories 
$\Ver_{p^n}$ over $\bf k$ defined in \cite{Benson/Etingof/Ostrik}. 
Namely, let $\mathcal T_p:={\rm Tilt}SL_2(\bf k)$ be the category 
of tilting modules over $SL_2(\bf k)$. Let $T_i\in \mathcal T_p$ 
be the tilting module for $SL_2(\bf k)$ with highest weight $i$. 
The module $T_{p^{n}-1}$ generates a tensor ideal
$\mathcal I_n\subset  \mathcal T_p$ spanned by  
$T_i$ for $i\ge p^n-1$. We define $\mathcal T_{n,p}$  to be the 
quotient category $\mathcal T_p/{\mathcal I}_n$. Then $\Ver_{p^n}$ 
is the abelian envelope of $\mathcal T_{n,p}$, i.e., the unique abelian 
symmetric tensor category containing $\mathcal T_{n,p}$ such that 
faithful symmetric monoidal functors out of $\mathcal T_{n,p}$
into abelian symmetric tensor categories uniquely factor through $\Ver_{p^n}$. 

More concretely, 
$\Ver_{p^n}$ is the category $R-$mod of finite dimensional modules over the algebra 
$R:={\rm End}(\oplus_{i=p^{n-1}-1}^{p^n-2}T_i)$,\footnote{It 
does not matter whether to take endomorphisms in 
$\mathcal T_p$ or $\mathcal T_{n,p}$ -- the 
corresponding natural map of endomorphism rings is an isomorphism.} 
realized as the homotopy category of projective resolutions 
$P^\bullet$ in $R-$mod with the usual tensor product. Namely, 
it turns out that the tensor product of resolutions is a resolution 
(i.e., acyclic in negative degrees), there is a unit object, and the 
corresponding tensor category is rigid (with $T_i^*\cong T_i$) and 
equipped with a natural faithful symmetric monoidal functor 
$\mathcal T_{n,p}\to \Ver_{p^n}$ given by $P^\bullet \mapsto H^0(P^\bullet)$. 

Let $\bT_i$ be the image of $T_i$ in $\Ver_{p^n}$. In particular, 
we let $V=\bT_1$ be the image of the 2-dimensional irreducible 
representation $T_1$ of $SL_2(\bf k)$, also denoted by $V$ 
(these of course depend on $n$ but to lighten the notation we do not indicate this explicitly). 
Note that in both categories $\Lambda^2 V$ is the unit object and 
$\Lambda^iV=0$ for $i\ge 3$. 
Recall~\cite{Etingof:2018a,Etingof/Harman/Ostrik:2018a} 
that we have the Koszul complex 
$K^\bullet:=S^\bullet V\otimes \Lambda V$ in $\Ver_{p^n}$ 
(i.e., we use the symmetric power superscript as the cohomological degree). 
This complex may also be graded by total degree, which is preserved
by the differential. So it splits into a direct 
sum of complexes $K_m^\bullet$, $m\ge 0$: 
\[ 0\to S^{m-2}V\to S^{m-1}V\otimes V\to S^mV \to 0 \]
(where we agree that $S^jV=0$ if $j<0$). The map 
$S^{m-1}V\otimes V\to S^mV$ in this complex is induced 
by the multiplication map of the algebra $SV$, so it is surjective 
when $m\ne 0$. 

\begin{proposition}\label{pr:exac} 
If $1\le m\le p^n-2$ then the complex $K_m^\bullet$ is exact.
\end{proposition} 
\begin{proof} 
It suffices to show that for any 
$i\in [p^{n-1}-1,p^n-2]$ the sequence\pagebreak[3]
\begin{equation}\label{e1}
0\to {\rm Hom}_{\Ver_{p^n}}(S^{m}V,\bT_i)\to 
{\rm Hom}_{\Ver_{p^n}}(S^{m-1}V\otimes V,\bT_i)\to 
{\rm Hom}_{\Ver_{p^n}}(S^{m-2}V,\bT_i) \to 0
\end{equation}
is exact. This sequence can be rewritten as 
{\small
\begin{equation}\label{e2}
0\to {\rm Hom}_{\Ver_{p^n}}(V^{\otimes m},\bT_i)^{S_{m}}\to 
{\rm Hom}_{\Ver_{p^n}}(V^{\otimes m},\bT_i)^{S_{m-1}}\to 
{\rm Hom}_{\Ver_{p^n}}(V^{\otimes m-2},\bT_i)^{S_{m-2}} \to 0. 
\end{equation}
}
By Theorem~4.2 of \cite{Benson/Etingof/Ostrik}, 
sequence~\eqref{e2} can be rewritten as 
{\small
\begin{equation}\label{e3}
0\to {\rm Hom}_{\mathcal{T}_{n,p}}(V^{\otimes m},T_i)^{S_{m}}\to 
{\rm Hom}_{\mathcal{T}_{n,p}}(V^{\otimes m},T_i)^{S_{m-1}}\to 
{\rm Hom}_{\mathcal{T}_{n,p}}(V^{\otimes m-2},T_i)^{S_{m-2}} \to 0. 
\end{equation} 
}
Now, if $1\le m\le p^n-2$, then by Proposition~3.5 
of~\cite{Benson/Etingof/Ostrik}, sequence \eqref{e3} can 
be rewritten as follows: 
\begin{equation}\label{e4}
0\to {\rm Hom}_{\mathcal{T}_{p}}(V^{\otimes m},T_i)^{S_{m}}\to 
{\rm Hom}_{\mathcal{T}_{p}}(V^{\otimes m},T_i)^{S_{m-1}}\to 
{\rm Hom}_{\mathcal{T}_{p}}(V^{\otimes m-2},T_i)^{S_{m-2}} \to 0, 
\end{equation} 
where $V$ now denotes the $2$-dimensional irreducible 
representation of $SL_2(\bf k)$. The Hom spaces in this 
sequence are just Homs between representations of $SL_2(\bf k)$. 
Thus sequence \eqref{e4} can be written as 
{\small
\begin{equation}\label{e5} 
0\to {\rm Hom}_{SL_2(\bf k)}(S^{m}V,T_i)\to 
{\rm Hom}_{SL_2(\bf k)}(S^{m-1}V\otimes V,T_i)\to 
{\rm Hom}_{SL_2(\bf k)}(S^{m-2}V,T_i) \to 0.
\end{equation} 
}

We will now use the following lemma.

\begin{lemma}\label{le:e1} 
For $m\le p^n-1$ one has ${\rm Ext}^1_{SL_2(\bf k)}(S^mV,T_i)=0$.
\end{lemma}

\begin{proof} 
Since $i\ge p^{n-1}-1$, it suffices to show that for any $j$, 
\[ \Ext^1_{SL_2(\bf k)}(S^mV,{\rm St}_{n-1}\otimes T_j)=0, \] 
where ${\rm St}_{n-1}:=T_{p^{n-1}-1}$ is the $(n-1)$st 
Steinberg module (note that it is self-dual). We have 
\[ \Ext^1_{SL_2(\bf k)}(S^mV,{\rm St}_{n-1}\otimes T_j)=
\Ext^1_{SL_2(\bf k)}(S^mV\otimes {\rm St}_{n-1},T_j). \]
By \cite{Benson/Etingof/Ostrik}, Lemma~3.3, 
$S^mV\otimes {\rm St}_{n-1}$ has a filtration whose 
successive quotients are tilting modules. Thus, since 
$\Ext^1(T_l,T_j)=0$, $S^mV\otimes {\rm St}_{n-1}$ is a direct sum of $T_i$, i.e., a 
tilting module. This implies the statement, using again that $\Ext^1(T_l,T_j)=0$. 
\end{proof} 

Now the exactness of \eqref{e5} follows from the fact that 
the sequence of $SL_2(\bf k)$-representations 
\[ 0\to S^{m-2}V\to S^{m-1}V\otimes V\to S^mV \to 0 \]
is exact (being a homogeneous part of the ordinary Koszul complex) 
and Lemma~\ref{le:e1}.  This completes the proof of Proposition~\ref{pr:exac}.
\end{proof} 

Let ${\rm q}=e^{\pi i/p^n}$. 

\begin{corollary}\label{cofpdim1}
\begin{enumerate}[\rm (i)]
\item For $m\le p^n-2$ we have 
\[ {\rm FPdim}(S^mV)=[m+1]_{\rm q}:=\frac{{\rm q}^{m+1}-{\rm q}^{-m-1}}{{\rm q}-{\rm q}^{-1}}\in
  \bR \]
and 
$ \dim(S^mV)=m+1\in {\bf k}$. 
\item The Jordan--H\"older multiplicities of the objects $S^mV$ are 
the decomposition numbers of tilting modules into Weyl modules computed 
in \cite{Tubbenhauer/Wedrich} (see \cite{Benson/Etingof/Ostrik}, Proposition 4.17). 
\end{enumerate}
\end{corollary} 

\begin{proof} (i) It follows from Proposition \ref{pr:exac} 
that 
\begin{align*}
{\rm FPdim}(S^mV)&=({\rm q}+{\rm q}^{-1}){\rm FPdim}(S^{m-1}V)-
{\rm FPdim}(S^{m-2}V), \\
\dim(S^mV)&=2\dim(S^{m-1}V)-\dim(S^{m-2}V).
\end{align*}
Thus the statement follows by induction, using that $S^0V=\one,\ S^1V=V$. 

(ii) This follows from (i), using \cite{Benson/Etingof/Ostrik}
Theorem 4.5(iv) and 
Propositions~4.12, 4.16.
\end{proof} 

Recall (\cite{Benson/Etingof/Ostrik}) that 
$\Ver_{p^n}$ has exactly two invertible objects up to isomorphism 
for $p>2$ and exactly one (the unit) for $p=2$. 
For $p>2$ let $\psi$ be the unique non-trivial invertible object 
of $\Ver_{p^n}$ (generating the category of supervector spaces). 
If $p=2$, we agree that $\psi=\one$. 

\begin{corollary}\label{cofpdim2}
\begin{enumerate}[\rm (i)]
\item $S^{p^n-2}V=\psi$.
\item $S^{p^n-3}V=V\otimes \psi$.
\item $S^jV=0$ for all $j>p^n-2$.  
\end{enumerate}
\end{corollary} 

\begin{proof} 
(i) By Corollary \ref{cofpdim1}, we have ${\rm FPdim}(S^{p^n-2}V)=1$, 
which implies that $S^{p^n-2}V$ is invertible (see \cite{EGNO}, Ex. 4.5.9). For $p=2$ this implies that 
$S^{2^n-2}V=\one$, and for $p>2$ that 
$S^{p^n-2}V=\psi$ (as $S^{p^n-2}V\in \Ver_{p^n}^-$ since $p^n-2$ is odd). 

(ii) Similarly, by Corollary \ref{cofpdim1},
${\rm FPdim}(S^{p^n-3}V)={\rm q}+{\rm q}^{-1}<2$, so $S^{p^n-3}V$ is simple. But by the results of \cite{Benson/Etingof/Ostrik}, the only object $X\in \Ver_{p^n}$ of Frobenius-Perron dimension ${\rm q}+{\rm q}^{-1}$ such that $\psi$ is a quotient of $X\otimes V$ is $X\cong V\otimes \psi$. Thus $S^{p^n-3}V\cong V\otimes \psi$. 

(iii) The map $S^{p^n-3}V\to S^{p^n-2}V\otimes V$ 
corresponds to the surjective map 
$S^{p^n-3}V\otimes V\to S^{p^n-2}V$, which is nonzero by (i). 
Hence by (ii) it is an isomorphism. Thus the morphism 
$S^{p^n-2}V\otimes V\to S^{p^n-1}V$ must be $0$ 
(as $K_{p^n-1}^\bullet$ is a complex). But this map is surjective, 
so $S^{p^n-1}V=0$. This implies the statement. 
\end{proof} 

\begin{remark} In particular, this implies that 
\[ \sum_{m=0}^\infty \dim(S^mV) z^m=(1-z)^{p^n-2}\in \bk[[z]]. \]
Also we clearly have 
$$
\sum_{m=0}^\infty \dim(\Lambda^mV) z^m=1+2z+z^2=(1+z)^2\in \bk[[z]].
$$ 
Thus the $p$-adic dimensions of $V$ defined 
in~\cite{Etingof/Harman/Ostrik:2018a} are as follows: 
\[ {\rm Dim}_-(V)=2\in \bZ_p,\ {\rm Dim}_+(V)=2-p^n\in \bZ_p.
\] 
Similarly, we get 
\begin{equation}\label{dimsv}
\sum_{m=0}^\infty{\rm FPdim}(S^mV)z^m=\frac{1+z^{p^n}}{(1-{\rm q}z)(1-{\rm q}^{-1}z)}.
\end{equation}
\end{remark} 

We also obtain 

\begin{corollary}\label{co:koszul}
\begin{enumerate}[\rm (i)]
\item The Koszul complex $K^\bullet$ 
is exact in all degrees except $0$ and $p^n-2$. Moreover 
$H^0(K^\bullet)=\one$ sitting in total degree $0$ and 
$H^{p^n-2}(K^\bullet)=\psi$ sitting in total degree $p^n$. 

\item The algebra $SV$ is $(p^n-2,2)$-Koszul and 
the algebra $\Lambda V$ is $(2,p^n-2)$-Koszul in the sense 
of Brenner, Butler and King~\cite{Brenner/Butler/King:2002a} 
(see~\cite{Etingof:2018a}, Definition 5.3). 
\end{enumerate}
\end{corollary} 

\begin{corollary}\label{cofrob} 
The algebra $SV$ in $\Ver_{p^n}$ is Frobenius. 
\end{corollary}

\begin{proof} 
Assume the contrary, and let $k$ be the largest integer 
such that the left kernel of the pairing 
$S^kV\otimes S^{p^n-2-k}V\to S^{p^n-2}V=\psi$ is nonzero. 
Denote this kernel by $N$. 
Then the composite map 
$
N\otimes V\to S^kV\otimes V\to S^{k+1}V
$ 
is zero. 
Thus the composite map 
$
N\to S^kV\to S^{k+1}V\otimes V
$ 
is zero. 
But by Proposition~\ref{pr:exac}, the map $S^kV\to S^{k+1}V\otimes V$ 
is injective. Thus $N=0$, a contradiction. 
\end{proof} 

\begin{remark} 
Recall (\cite{Benson/Etingof/Ostrik}, Subsection~4.4)
that the category $\Ver_{p^n}=\Ver_{p^n}(\bf k)$ lifts 
to a semisimple braided (non-symmetric) category $\Ver_{p^n}({\bf K})$ over 
a field ${\bf K}$ of characteristic zero, 
corresponding to the quantum group $SL_2^{-{\rm q}}$ where 
${\rm q}$ is a primitive root of unity of order $2p^n$ in ${\bf K}$. 
In $\Ver_{p^n}({\bf K})$ we have the quantum symmetric algebra 
$S_{-\rm q}V$, which is a lift of 
$SV$ over ${\bf K}$ and is also Frobenius $(p^n-2,2)$-almost 
Koszul (see \cite{Etingof:2018a}, Subsection 5.5). In particular, we have 
the quantum Koszul complex $S^\bullet_{-\rm q}V\otimes \Lambda_{-\rm q}V$
in $\Ver_{p^n}({\bf K)}$ which is a flat deformation of the Koszul complex 
$S^\bullet V\otimes \Lambda V$ and has the cohomology as described in 
Corollary~\ref{co:koszul}. 
\end{remark} 

Corollary~\ref{co:koszul} allows us to construct an injective resolution $Q_\bullet$:
\[ Q_0\to Q_1\to Q_2\to\cdots \]
 of the augmentation $\Lambda V$-module $\one$ by free 
$\Lambda V$-modules, which is periodic with period $2^n-1$ 
for $p=2$ and antiperiodic with period $p^n-1$ for $p>2$ 
(where antiperiodic means that it multiplies by $\psi$ when 
shifted by this period; in particular, it is $2(p^n-1)$-periodic). 
Namely, for $0\le i\le p^n-2$ we have 
$Q_{2r(p^n-1)+m}=S^mV\otimes \Lambda V$, and 
$Q_{(2r+1)(p^n-1)+m}=S^mV\otimes \psi\otimes \Lambda V$. 

\begin{remark} If $p^n=2$ (i.e., $p=2$, $n=1$) then $V=0$, so the Koszul complex reduces to $\one$ sitting in degree $0$ and hence does not fit the above general pattern; but we will not consider this trivial case. If $p^n=3$ (i.e., $p=3,n=1$) then $V=\psi$, so $\Lambda^3\psi\ne 0$ and hence the Koszul complex $S^\bullet V\otimes \Lambda V$ still does not fit the general pattern (in fact, in this case $\Ver_{p^n}={\rm Supervec}$, so the Koszul complex is exact except in degree $0$).  However, now this can be remedied by a slight modification of the definition. Namely, let \
$\Lambda_{\rm tr} V$ be the quotient of $\Lambda V$ by $\Lambda^3 V$ (forcing the desired equality $\Lambda^3 V=0$). Then we have the truncated Koszul complex 
$K_{\rm tr}^\bullet:=S^\bullet V\otimes \Lambda_{\rm tr}V$ which is easily shown to have the same properties as the usual Koszul complex $K^\bullet$ for $p^n>3$. Thus if $p^n=3$ then, abusing terminology and notation, by $\Lambda V$ we will mean $\Lambda_{\rm tr}V$, and by the Koszul complex the truncated Koszul complex; then the above results will also apply to this case. 
\end{remark} 

As an application let us compute the multiplicities of the unit object in the symmetric powers of $V$ for $p=2$. 

\begin{proposition}\label{pr:multi2} 
If $p=2$ then $[S^mV:\one]=0$ if $m$ is odd and 
$[S^mV:\one]=1$ if $m$ is even. Thus $[SV:\one]=2^{n-1}$. 
\end{proposition}

\begin{proof} 
Note that for $X\in \Ver_{2^{n}}$ the multiplicity 
$[X:\one]$ of $\one$ in $X$ is ${\rm Tr}({\rm FPdim}(X))/2^{n-1}$
(the trace of the algebraic number in the field $\bQ({\rm q}+{\rm q}^{-1}))$ where ${\rm  q}:=e^{\pi i/2^{n}}$; this follows since by 
\cite{Benson/Etingof:2019a}, 
${\rm Tr}{\rm FPdim}(X)=0$ for any nontrivial simple $X\in \Ver_{2^{n+1}}$. So 
we have 
$$
\sum_m [S^mV:\one]z^m=\frac{1}{2^{n-1}}{\rm Tr}\left(\frac{1+z^{2^n}}{(1-{\rm q}z)(1-{\rm q}^{-1}z)}\right).
$$
Thus the result follows from the following lemma. 

\begin{lemma} 
$\displaystyle
\frac{1}{2^{n-1}}{\rm Tr}\left(\frac{1+z^{2^n}}{(1-{\rm q}z)(1-{\rm q}^{-1}z)}\right)=\frac{1-z^{2^n}}{1-z^2}=\sum_{j=0}^{2^{n-1}-1}z^{2j}.
$
\end{lemma} 

\begin{proof} We have 
$$
\frac{1}{2^{n-1}}{\rm Tr}\left(\frac{1+z^{2^n}}{(1-{\rm q}z)(1-{\rm q}^{-1}z)}\right)=
\frac{1}{2^{n-1}}\sum_{k=1}^{2^{n-1}}\frac{1+z^{2^n}}{(1-{\rm q}^{2k-1}z)(1-{\rm q}^{-2k+1}z)}.
$$
This is the unique polynomial $h(z)\in \bQ[z]$ 
of degree $2^n-2$ such that $h({\rm q}^j)=\frac{2}{1-{\rm q}^{2j}}$
for any odd number $j$. But the polynomial $\frac{1-z^{2^n}}{1-z^2}$ 
satisfies these conditions, hence the result. 
\end{proof} 
 
This completes the proof of Proposition~\ref{pr:multi2}.
\end{proof} 

\begin{remark} Another proof of Proposition \ref{pr:multi2} is obtained by applying 
Proposition 4.16 and Theorem 4.42 of \cite{Benson/Etingof/Ostrik}. Namely, 
$[S^iV:\one]$ is an entry of the decomposition matrix of $\Ver_{2^n}$, so it is $0$ 
if $i$ is odd and $1$ if $i$ is even. This follows since the descendants 
of the number $2^n-1$ are exactly all the odd numbers between $1$ and $2^n-1$.  
\end{remark} 

\section{\texorpdfstring
{Ext computations}
{Ext computations}}

\subsection{\texorpdfstring
{Ext computations for $p=2$}
{Ext computations for  𝑝 = 2}}
Consider now the case $p=2$. In this case, we can use the 
resolution $Q_\bullet$ to give the following recursive procedure 
of computation of the additive structure of the cohomology 
$\Ext^\bullet_{\Ver_{2^{n+1}}}(\one,X)$ (for indecomposable $X$). 

We will denote the generating 
object of $\Ver_{2^{k+1}}$ by $X_k$ and recall that 
$\Ver_{2^{k+2}}^+$ is the category of $\Lambda X_k$-modules in 
$\Ver_{2^{k+1}}$. Also the resolution $Q_\bullet$ in 
$\Ver_{2^{k+1}}$ will be denoted by $S^\bullet X_{k}[y_{k+1}]\otimes \Lambda X_k$, 
where $y_{k+1}$ is a variable of degree $2^{k+1}-1$ for $k\ge 0$. This is justified by this 
resolution being periodic with period $2^{k+1}-1$. Also if $Y^\bullet,Z^\bullet$ are complexes in an abelian category $\mathcal A$ then by $\Ext^m(Z^\bullet,Y^\bullet)$ we will mean ${\rm Hom}(Z^\bullet,Y^\bullet[m])={\rm Hom}(Z^\bullet,Y^{\bullet+m})$ with Hom taken in the derived category $D(\mathcal{A})$. 

Recall that $\Ver_{2^{n+1}}=\Ver_{2^{n+1}}^+\oplus \Ver_{2^{n+1}}^-$. If $X\in \Ver_{2^{n+1}}^-$, then
$\Ext^\bullet_{\Ver_{2^{n+1}}}(\one,X)$ is zero. Thus, it suffices to compute 
$\Ext^\bullet_{\Ver_{2^{n+1}}^+}(\one,X)$ for $X\in \Ver_{2^{n+1}}^+$. 
In that case, we have 
\begin{align}\label{recu}
\Ext^\bullet_{\Ver_{2^{n+1}}^+}(\one,X)&\cong \Ext^\bullet_{\Lambda X_{n-1}}(\one,X)\cong \Ext^\bullet_{\Lambda X_{n-1}}(\one,Q_\bullet\otimes X)\\ \nonumber
&\cong\Ext^\bullet_{\Lambda X_{n-1}}(\one,S^\bullet X_{n-1}[y_n]
\otimes\Lambda X_{n-1}\otimes X) \\ \nonumber
&\cong\Ext^\bullet_{\Ver_{2^n}}(\one,S^\bullet X_{n-1}[y_n]
\otimes X) \\ \nonumber
&\cong\Ext^\bullet_{\Ver_{2^n}^+}(\one,(S^\bullet X_{n-1}[y_n]\otimes X)^+)
\end{align}
where the superscript $+$ means that we are taking the part 
lying in $\Ver_{2^n}^+$, and in the last two expressions 
$X$ is regarded as an object of $\Ver_{2^n}$ 
using the corresponding forgetful functor $\Ver_{2^{n+1}}^+\to \Ver_{2^n}$.  
Here for the penultimate isomorphism we invoked the Shapiro lemma, using that 
the $\Lambda X_{n-1}$-module $S^k X_{n-1}[y_n] \otimes\Lambda X_{n-1}\otimes X$ is free and therefore coinduced (as $\Lambda X_{n-1}$ is a Frobenius algebra). 

Thus we get a recursion expressing of $\Ext^\bullet_{\Ver_{2^{n+1}}^+}(\one,X)$ in terms of $\Ext^\bullet_{\Ver_{2^{n}}^+}(\one,X')$. While this is a good news, unfortunately $X'$ 
is not an object any more but rather a complex of objects finite in the negative direction. Luckily, the same calculation applies if $X$ is such a complex, i.e., an object of the derived category $D^+(\Ver_{2^{n+1}}^+)$ of $\Ver_{2^{n+1}}^+$, which allows us to iterate this construction. Namely, for an object $X\in D^+(\Ver_{2^{n+1}}^+)$, let \[ E_n(X):=\underline{{\rm Hom}}_{\Lambda X_{n-1}}
(\one,S^\bullet X_{n-1}[y_n]\otimes \Lambda X_{n-1}\otimes X)^+
=(S^\bullet X_{n-1}[y_n]\otimes X)^+ \]
(the internal Hom taken in the category $\Ver_{2^n}$). 
This gives an additive functor 
$$
E_n\colon D^+(\Ver_{2^{n+1}}^+)\to 
D^+(\Ver_{2^n}^+).
$$ 

\begin{lemma}\label{dif0}
If $X\in \mathcal \Ver_{2^n}^+$ (i.e., a trivial $\Lambda X_{n-1}$-module) 
then the differential in the complex $E_n(X)$ is zero. 
\end{lemma} 

\begin{proof} It is easy to see that for a finite dimensional vector space $V$ over $\bk$, 
the differential on ${\rm Hom}_{\Lambda V}(\bk, S^\bullet V\otimes \Lambda V)=S^\bullet V$ 
induced by the Koszul differential on $S^\bullet V\otimes \Lambda V$ is zero. The lemma is 
a straightforward generalization of this fact. 
\end{proof} 

\begin{corollary}\label{rec1} Suppose $X\in \Ver_{2^n}$. Then we have an isomorphism 
\begin{align*}
\Ext^\bullet_{\Ver_{2^{n+1}}^+}(\one,X)&\cong \bigoplus_{i\ge 0}
\Ext^{\bullet-i}_{\Ver_{2^n}}(\one,S^iX_{n-1}[y_n]\otimes X)\\
&=\bigoplus_{i\ge 0}\Ext^{\bullet-i}_{\Ver_{2^n}^+}(\one,(S^iX_{n-1}[y_n]\otimes X)^+).
\end{align*}
This isomorphism maps the grading induced by the grading on $\Lambda X_{n-1}$
to the grading defined by $\deg(X_{n-1})=1$, 
$\deg(y_n)=2^n-1$ (i.e., it coincides with the cohomological grading). 
\end{corollary} 

\begin{proof} Follows immediately from Lemma \ref{dif0}.
\end{proof} 

\begin{remark} 
Corollary \ref{rec1} does not quite give a recursion to compute the Ext groups, since the object $(S^iX_{n-1}\otimes X)^+$ may not belong to $\Ver_{2^{n-1}}$ (i.e., it may carry a nontrivial action of $\Lambda X_{n-2}$). However, it has some useful consequences given below. 
\end{remark} 

Now recall that $\Ver_2={\rm Vec}$ and $\Ver_{2^2}^+$ is the category of ${\bf k}[\xi]$-modules where $\xi^2=0$. 
Define a functor $E_1\colon D^+(\Ver_{2^2}^+)\to D^+({\rm Vec})$ by 
\[ E_1(X):={\rm Hom}_{{\bf k}[\xi]}({\bf k},{\bf k}[y_1,\xi]\otimes X)
={\bf k}[y_1]\otimes X, \]
with the differential 
$ d(y_1^m\otimes x)=y_1^{m+1}\otimes \xi x+y_1^m\otimes dx$. 
We thus obtain the following proposition. 

\begin{proposition} 
We have a natural isomorphism
\[ \Ext^\bullet_{D^+(\Ver_{2^{n+1}}^+)}(\one,X)\cong \Ext^\bullet_{D^+(\Ver_{2^n}^+)}(\one,E_n(X)) \]
for $n\ge 2$, and 
\[ \Ext^\bullet_{D^+(\Ver_{2^2}^+)}(\one,X)\cong \Ext^\bullet_{D^+({\rm Vec})}(\one,E_1(X)). \]
\end{proposition} 

This implies the following corollary. Let 
\[ E=E_1\circ\dots\circ E_n\colon D^+(\Ver_{2^{n+1}}^+)\to D^+({\rm Vec}). \]

\begin{corollary} 
We have a linear natural isomorphism 
\[ \Ext^\bullet_{D^+(\Ver_{2^{n+1}}^+)}(\one,X)=H^\bullet(E(X)). \]
\end{corollary} 

The complex of vector spaces $E(X)$ has the following structure: 
\begin{equation}\label{eofx}
E(X)=(S^\bullet X_1\otimes\dots\otimes(S^\bullet X_{n-2} 
\otimes 
(S^\bullet X_{n-1}\otimes X)^+)^+\dots)^+[y_1,y_2,\dots,y_n], 
\end{equation}
and it is easy to see that the differential is linear over ${\bf k}[y_1,\dots,y_n]$, since multiplication by $y_i$ is induced by the shift in the corresponding periodic resolution.  
Thus we get
 
\begin{proposition}\label{finge} 
For any $X\in \Ver_{2^{n+1}}$, $\Ext^\bullet_{\Ver_{2^{n+1}}}(\one,X)$ 
is a graded finitely generated module over ${\bf k}[y_1,\dots,y_n]$. 
\end{proposition}

In particular, we get that 
\[ \Ext^\bullet_{\Ver_{2^{n+1}}}(\one,\one)=
\Ext^\bullet_{\Ver_{2^{n+1}}^+}(\one,\one)=H^\bullet(E(\one)), \]
where 
\[ E(\one)=(S^\bullet X_1\otimes\dots\otimes(S^\bullet X_{n-2}
\otimes (S^\bullet X_{n-1})^+)^+\dots)^+[y_1,y_2,\dots,y_n]. \]
Note that we have $1\in E(\one)$ and $d(1)=0$, so we obtain a 
natural linear map 
\[ \phi\colon {\bf k}[y_1,\dots,y_n]\to 
\Ext^\bullet_{\Ver_{2^{n+1}}}(\one,\one). \] 

\begin{proposition}\label{homomo} For $1\le i\le n$ 
multiplication by $y_i$ on $\Ext^\bullet_{\Ver_{2^{n+1}}}(\one,X)$ 
coincides with the cup product with $\phi(y_i)$. In particular, 
$\phi$ is an algebra homomorphism. 
\end{proposition} 

\begin{proof} The proof is by induction in $n$. For $i<n$ the statement follows 
from the inductive assumption. For $i=n$, we see that the cup product with 
$\phi(y_n)$ can be realised as the Yoneda product ($=$ concatenation) 
with the Koszul complex $K^\bullet$, which represents the class 
$\phi(y_n)$ in the Yoneda realization of $\Ext$. This proves the 
first statement. The second statement
then follows since $\phi(ab)=(ab)\cdot 1=a\cdot (b\cdot 1)=a\cdot \phi(b)=\phi(a)\phi(b)$. 
\end{proof} 

\begin{proposition}\label{injec} 
For $X\in \Ver_{2^{n}}^+$ the natural map 
\[ \Ext^\bullet_{\Ver_{2^n}}(\one,X)[y_n]\to
  \Ext^\bullet_{\Ver_{2^{n+1}}}(\one,X) \] 
is an injective morphism of ${\bf k}[y_1,...,y_n]$-modules 
which is also a morphism of algebras for $X=\one$. 
\end{proposition} 

\begin{proof} This follows from the isomorphism 
$$
\Ext^\bullet_{\Ver_{2^{n+1}}}(\one,X)\cong \Ext^\bullet_{\Ver_{2^{n}}}(\one,(S^{\rm even} X_{n-1}\otimes X)[y_n])
$$
 since $\one$ is a direct summand of $S^{\rm even} X_{n-1}$. 
\end{proof} 

\begin{proposition}\label{byzero} Let $U\in \Ver_{2^n}$ and 
$X:=U\otimes \Lambda X_{n-1}\in \Ver_{2^{n+1}}^+$ be a free $\Lambda X_{n-1}$-module. Then 
$y_n$ acts on $\Ext^\bullet_{\Ver_{2^{n+1}}}(\one,X)$ by zero. 
\end{proposition} 

\begin{proof} By the Shapiro lemma we have 
$$
\Ext^\bullet_{\Ver_{2^{n+1}}}(\one,X)\cong \Ext^\bullet_{\Lambda X_{n-1}}(\one,X)\cong
\Ext^\bullet_{\Ver_{2^n}}(\one,U).
$$
Therefore, the group $\bG_m$ scaling $X_{n-1}$ acts trivially 
on $\Ext^\bullet_{\Ver_{2^{n+1}}}(\one,X)$. So the statement follows, as 
$y_n$ has degree $2^n-1$ with respect to this action. 
\end{proof} 

Let $S\subset \lbrace{1,...,n-1\rbrace}$ and $X_S:=\cotimes_{i\in S}X_i$ 
be the simple object of $\Ver_{2^{n}}$ attached to $S$ in \cite{Benson/Etingof:2019a}. 

\begin{proposition}\label{pr:tors} 
\begin{enumerate}[\rm (i)]
\item If $i\in S$ and $Y\in \Ver_{2^{n+1}}^+$ then multiplication by $y_i$ acts by zero 
on $\Ext^\bullet_{\Ver_{2^{n+1}}^+}(Y,X_S)$. 
Hence $\Ext^\bullet_{\Ver_{2^{n+1}}^+}(Y,X_S)$ is a torsion module 
over ${\bf k}[y_1,...,y_n]$ unless $S=\varnothing$ (i.e., $X_S=\one$). 
\item
 The annihilator of $\Ext^\bullet_{\Ver_{2^{n+1}}^+}(X_S,X_S)$ in ${\bf k}[y_1,...,y_n]$ 
is generated by $y_i$ with $i\in S$. 
\end{enumerate}
\end{proposition}

\begin{proof} (i) The proof is by induction in $n$. The base is clear, so we just have to 
justify the induction step. We have 
$$
\Ext^\bullet_{\Ver_{2^{n+1}}^+}(Y,X_S)\cong
\Ext^\bullet_{\Ver_{2^{n}}^+}(\one,(X_S\otimes S^\bullet X_{n-1}\otimes Y^*)^+[y_n]). 
$$
If $n-1\notin S$ then $X_S\in \Ver_{2^n}^+$ so this can be written as 
$$
\Ext^\bullet_{\Ver_{2^{n+1}}^+}(Y,X_S)\cong
\Ext^\bullet_{\Ver_{2^{n}}^+}((Y\otimes S^\bullet X_{n-1}[y_n]^*)^+,X_S)
$$
and the statement follows from the inductive assumption. On the other hand, if 
$n-1\in S$ then setting $S'=S\setminus \lbrace n-1\rbrace$, we have 
$X_S=X_{S'}\otimes X_{n-1}$. So we get 
$$
\Ext^\bullet_{\Ver_{2^{n+1}}^+}(Y,X_S)
\cong\Ext^\bullet_{\Ver_{2^{n}}^+}(\one,X_{S'}\otimes X_{n-1}\otimes (S^\bullet X_{n-1}\otimes Y^*)^-[y_n]),
$$
where the superscript minus sign means that we are taking the part lying in $\Ver_{2^n}^-$. 
But $
(S^\bullet X_{n-1}\otimes Y^*)^-=X_{n-1}\otimes W^\bullet
$
for some $W^\bullet\in \Ver_{2^n}^+$, and $X_{n-1}\otimes X_{n-1}=\Lambda X_{n-2}$ (and some differential on the tensor product whose exact form is not important for this argument). Thus we get 
\begin{align*}
\Ext^\bullet_{\Ver_{2^{n+1}}^+}(Y,X_S)&\cong
\Ext^\bullet_{\Ver_{2^{n}}^+}(\one,X_{S'}\otimes X_{n-1}
\otimes X_{n-1}\otimes W^\bullet[y_n])\\
&\cong\Ext^\bullet_{\Lambda X_{n-2}}(\one,X_{S'}\otimes 
\Lambda X_{n-2}\otimes W^\bullet[y_n])\\
&\cong\Ext^\bullet_{\Ver_{2^{n-1}}}(W^\bullet[y_n]^*,X_{S'}).
\end{align*}
So by Proposition \ref{byzero} the element $y_{n-1}$ acts on this 
space by zero, and the statement again follows from the inductive assumption. 

(ii) By (i) the annihilator is at least as big as claimed, 
and we only need to show that it is not bigger. This is shown 
again by induction in $n$. The base is again easy so we only need to 
do the induction step. If $n-1\notin S$ then 
by Proposition~\ref{injec} we have an inclusion 
$\Ext^\bullet_{\Ver_{2^n}^+}(X_S,X_S)[y_n]\to
\Ext^\bullet_{\Ver_{2^{n+1}}^+}(X_S,X_S)$, 
so the result follows from the inductive assumption for $n-1$. 
On the other hand, if $n-1\in S$ then $X_S=X_{S'}\otimes X_{n-1}$
so 
\begin{align*}
\Ext^\bullet_{\Ver_{2^{n+1}}^+}(X_S,X_S)&\cong
\Ext^\bullet_{\Lambda X_{n-1}}(X_{S'},X_{S'}\otimes\Lambda X_{n-2}) \\
&\cong\Ext^\bullet_{\Lambda X_{n-2}}(X_{S'},X_{S'}\otimes 
\Lambda X_{n-2}\otimes S^{\rm even}X_{n-1}) \\
&\cong\Ext^\bullet_{\Ver_{2^{n-1}}}(X_{S'},X_{S'}\otimes S^{\rm even}X_{n-1}), 
\end{align*}
which contains $\Ext^\bullet_{\Ver_{2^{n-1}}}(X_{S'},X_{S'})=
\Ext^\bullet_{\Ver_{2^{n-1}}^+}(X_{S'},X_{S'})$ as a direct summand 
as $S^{\rm even}X_{n-1}$ contains $\one$ as a direct summand. 
Thus the result again follows from the inductive assumption (this time for $n-2$).  
\end{proof} 

\begin{corollary}\label{co:recu} The rank $r_n$ of the module 
$\Ext^\bullet_{\Ver_{2^{n+1}}}(\one,\one)$ over ${\bf k}[y_1,...,y_n]$ satisfies the equality 
$r_n=r_{n-1}[SX_{n-1}:\one]$. 
\end{corollary} 
\begin{proof} In view of the isomorphism \eqref{recu} applied to $X=\one$, this follows from Proposition \ref{pr:tors}\,(i). 
\end{proof} 

\begin{corollary}
We have $r_n=2^{\frac{n(n-1)}{2}}$. 
\end{corollary} 

\begin{proof} This follows from Corollary~\ref{co:recu} and 
Proposition~\ref{pr:multi2}, using that $r_1=1$. 
\end{proof} 

Recall that the algebra $\Ext^\bullet_{\Ver_{2^{n+1}}}(\one,\one)$ has a $\bZ$-grading coming from the grading on $\Lambda X_{n-1}$, where $y_n$ has degree $2^{n}-1$. Define the field $F_n:={\bf k}(y_1,...,y_{n-1})$, and let 
$r_n(v)$ be the Poincar\'e polynomial of 
$\Ext^\bullet_{\Ver_{2^{n+1}}}(\one,\one)\otimes_{{\bf k}[y_1,...,y_{n-1}]}F_n$ as a module over the algebra $F_n[y_n]$. 
Then the above arguments yield 

\begin{corollary} $\Ext^\bullet_{\Ver_{2^{n+1}}}(\one,\one)\otimes_{{\bf k}[y_1,...,y_{n-1}]}F_n$ is a free $F_n[y_n]$-module, and for $n\ge 2$ 
$$
r_n(v)=2^{\frac{(n-1)(n-2)}{2}}\frac{1-v^{2^{n}}}{1-v^2}=2^{\frac{(n-1)(n-2)}{2}}\sum_{j=0}^{2^{n-1}-1}v^{2j}. 
$$
\end{corollary} 

This agrees with Conjecture \ref{mainconj}. Also the formula $r_n=2^{\frac{n(n-1)}{2}}$ is now  obtained by evaluating $r_n(v)$ at $v=1$.  

\begin{remark} 
As stated in Conjecture \ref{mainconj}, we expect that 
moreover  $\Ext^\bullet_{\Ver_{2^{n+1}}}(\one,\one)$ is a free 
${\bf k}[y_1,...,y_{n}]$-module (even without localization in $y_1,...,y_{n-1}$). 
\end{remark}

More generally, for every object $X\in \Ver_{2^{n+1}}^+$ we obtain upper bounds 
for the Poincar\'e polynomials of generators of  $\Ext^\bullet_{\Ver_{2^{n+1}}}(\one,X)$, 
\[ r_n(X,z,v):=\sum_{i,j=0}^\infty  z^iv^j\dim\left(\Ext^\bullet_{\Ver_{2^{n+1}}}(\one,X)
\middle/{\textstyle\sum_{k=1}^n} \im(y_k)\right)^{i,j} \]
where $i$ is the cohomological degree and $j$ is the $v$-degree. 

\begin{proposition}\label{bound} For $n\ge 2$ we have
$ r_n(X,z,v)\le r_n^*(X,z,v)$,
in the sense that each coefficient on the left is less than or equal to the
corresponding coefficient on the right, where 
{\small
\[ r_n^*(X,z,v):=
\frac{1}{2^{n-1}}{\rm Tr}\left({\rm FPdim}(X)\frac{1+(zv)^{2^{n}}}{(1-{\rm q}zv)(1-{\rm q}^{-1}zv)}\prod_{j=2}^{n-1}
\frac{1+z^{2^j}}{(1-{\rm q}^{2^{j-1}}z)(1-{\rm
    q}^{-2^{j-1}}z)}\right). \] 
}
In particular, all generators have degree $\le 2^{n+1}-2(n+1)$.
\end{proposition} 

\begin{proof} 
The bound for $r_n(X,z,v)$ 
follows from the form of $E(X)$ given in \eqref{eofx} and formula \eqref{dimsv} by a direct computation. This implies the bound on the degree of generators, since the degree of $r_n^*$ with respect to $z$ is 
$2^{n+1}-2(n+1)$. 
\end{proof} 

In particular, for $X=\one$ we get 

\begin{corollary}\label{bound1}
$\Ext^\bullet_{\Ver_{2^{n+1}}}(\one,\one)$ is a finitely generated module 
over ${\bf k}[y_1,...,y_n]$ with Poincar\'e polynomial of generators 
$r_n(z,v)\le r_n^*(z,v)$, where 
\[ r_n^*(z,v):=
\frac{1}{2^{n-1}}{\rm Tr}\left(\frac{1+(zv)^{2^{n}}}{(1-{\rm q}zv)(1-{\rm q}^{-1}zv)}\prod_{j=2}^{n-1}
\frac{1+z^{2^j}}{(1-{\rm q}^{2^{j-1}}z)(1-{\rm q}^{-2^{j-1}}z)}\right). \] 
 In particular, 
all the generators have degree at most $2^{n+1}-2(n+1)$, and 
there is exactly one generator of that degree.  
Moreover, the Poincar\'e polynomial of generators is palindromic, i.e., satisfies the equation 
$P(z)=z^{2^{n+1}-n-1}P(z^{-1})$. 
\end{corollary} 

\begin{proof}
It only remains to show that the Poincar\'e polynomial of the generators 
is palindromic, which follows from the fact that the complex 
$E(\one)$ is self-dual. 
\end{proof} 

\begin{remark} Note that according to Conjecture \ref{mainconj}, the degree bound of Corollary~\ref{bound1} 
is expected to be sharp: the largest degree of a generator is expected to equal 
$2^{n+1}-2(n+1)$, with exactly one generator in that degree. 
On the other hand, the bound $r_n\le r_n^*$ is rather 
poor: we have $\log_2(r_n^*(1,1))\sim n^2$ as $n\to \infty$, while 
$\log_2(r_n(1,1))=\frac{n(n-1)}{2}$. This is not surprising, 
as this bound does not take into account the fact that the 
complex $E(\one)$ has a nontrivial differential for $n\ge 3$ and 
shrinks drastically when we compute its cohomology.  
\end{remark} 

\begin{example} 1. Let $n=2$. Then we have 
$E(\one)=(SX_1)^+[y_1,y_2]$. But $S^0X_1=S^2X_1$ $=\one$, $S^1X_1=X_1$, 
and all the other symmetric powers are zero. 
Thus, $(SX_1)^+={\bf k}\oplus {\bf k}w$, where $w$ has cohomological 
degree $2$. Also in this case it is easy to see that the differential in 
$E(\one)$ is zero (so the bound $r_2^*(z,v)=1+(zv)^2$ is sharp). Thus 
$\Ext^\bullet_{\Ver_{2^3}}(\one,\one)$ is a free ${\bf k}[y_1,y_2]$ module 
of rank $2$ with generators of degree $0$ and $2$, which agrees 
with the result of \cite{Benson/Etingof:2019a}. 

2. Let $n=3$. Let $S^i:=S^iX_2$. Then one can show by a direct 
computation that 
\[ S^0=\one,\ 
S^1=X_2,\  
S^2=[\one,X_1],\ 
S^3=X_1\otimes X_2,\ 
S^4=[X_1,\one],\ 
S^5=X_2,\ 
S^6=\one, \]
and all the other symmetric powers are zero (where 
$Y=[Y_1,...,Y_m]$ means that $Y$ is a uniserial 
object with composition series $Y_1,\dots,Y_m$, with head $Y_1$ and socle $Y_m$). 
Thus $\Ext^\bullet_{\Ver_{2^4}}(\one,\one)$ is isomorphic to
\[ \Ext^\bullet_{\Ver_{2^3}}(\one,\one)[0]\oplus
  \Ext^\bullet_{\Ver_{2^3}}(\one,[\one,X])[2] 
\oplus 
\Ext^\bullet_{\Ver_{2^3}}(\one,[X,\one])[4]\oplus
\Ext^\bullet_{\Ver_{2^3}}(\one,\one)[6], \]
where $X=X_1$ and the numbers in square brackets are degree shifts. Now, 
consider the portion of the long exact sequence 
\begin{equation}\label{longexa}
{\rm Hom}(\one,X)\to \Ext^1(\one,\one)\to 
\Ext^1(\one,[X,\one])\to \Ext^1(\one,X)\to \Ext^2(\one,\one), 
\end{equation} 
where $\Ext$ groups are taken in $\Ver_{2^3}$. It was shown in \cite{Benson/Etingof:2019a} that the Poincar\'e series 
of $\Ext^\bullet(\one,X)$ is $\frac{z}{1-z^3}$. Also we have 
$\dim\Ext^1(\one,[X,\one])\ge 2$ since we have two different 
nontrivial extensions of $\one $ by $[X,\one]$, namely 
$[\one\oplus X,\one]$ and $[\one,X,\one]$ (both indecomposable quotients of
the projective cover of $\one$ in $\Ver_{2^3}$). Thus the dimension of
$\Ext^1(\one,[X,\one])$ is two, and   
the sequence \eqref{longexa} looks like
$0\to \bold k\to \bold k^2\to \bold k\to \bold k$.
This implies that the last map in this sequence 
(the connecting homomorphism 
$\Ext^1(\one,X)\to \Ext^2(\one,\one)$) is zero. Since the map 
$\Ext^\bullet(\one,X)\to \Ext^{\bullet+1}(\one,\one)$ is linear 
over $\bold k[y_2]$, we see that this map is zero in all degrees
(as $\Ext^\bullet(\one,X)$ is a free $\bold k[y_2]$-module on one generator in degree $1$). Thus, 
$\Ext^\bullet(\one,[X,\one])\cong \Ext^\bullet(\one,X)\oplus 
\Ext^\bullet(\one,\one)$, so the Poincar\'e series of 
$\Ext^\bullet(\one,[X,\one])$ is $\frac{1+z}{(1-z)(1-z^3)}$. 

Now, the object $[X,\one,\one,X]$ is the projective cover of $X$. This implies that 
\[ \Ext^\bullet(\one,[\one,X])\cong
  \Ext^{\bullet+1}(\one,[X,\one]). \] 
Thus the Poincar\'e series 
of $\Ext^\bullet(\one,[\one,X])$ is $\frac{z+z^2}{(1-z)(1-z^3)}$. 
Altogether we obtain that the Poincar\'e series of 
$\Ext^\bullet_{\Ver_{2^4}}(\one,\one)$ is given by the formula 
\begin{align*}
h(z,v)&=\frac{(1+(vz)^6)(1+z^2)+(vz)^2(z+z^2)+(vz)^4(1+z)}{(1-z)(1-z^3)(1-(vz)^7)}\\
&=\frac{1+z^2+v^2z^3+(v^2+v^4)z^4+v^4z^5+v^6z^6+v^6z^8}{(1-z)(1-z^3)(1-(vz)^7)}.
\end{align*}
One can check directly that $\Ext^\bullet_{\Ver_{2^4}}(\one,\one)$ 
is a free module over $\bold k[y_1,y_2,y_3]$.  
Thus the Poincar\'e polynomial of its generators is 
$$
r_3(z,v)=1+z^2+v^2z^3+(v^2+v^4)z^4+v^4z^5+v^6z^6+v^6z^8. 
$$
On the other hand, it is easy to compute that 
\begin{align*}
r_3^*(z,v)&=1+(1+v^2)z^2+2v^2z^3+(v^2+v^4)z^4+2v^4z^5+(v^4+v^6)z^6+v^6z^8\\
&=r_3(z,v)+v^2(z^2+z^3)+v^4(z^5+z^6).
\end{align*}
This means that the differential in the complex 
$E(\one)/(y_1,y_2,y_3)$ acts as a rank $1$ operator between degrees 
$2\to 3$ and $5\to 6$ and otherwise acts by zero. In other words, when computing the cohomology of this complex, we kill two elements 
of cohomological degrees $2,3$ in $v$-degree $2$ and two elements of cohomological degrees $5,6$ in $v$-degree $4$. 

It is instructive to write down the complex $E(\one)$ explicitly. We have 
\[ E(\one)=M^+[y_1,y_2,y_3],\ M:=SX_1\otimes (SX_2)^+. \]
The components of $M$ are as follows (with $X:=X_1$):
\begin{gather*}
M^0=\one,\quad 
M^1=X,\quad 
M^2=\one\oplus [\one,X],\quad 
M^3=X\otimes [\one,X]=[X,\one,\one],\\
M^4=[\one,X]\oplus [X,\one],\quad
M^5=X\otimes [X,\one]=[\one,\one,X],\\ 
M^6=\one\oplus [X,\one],\quad 
M^7=X,\quad 
M^8=\one.
\end{gather*}
Thus, $E(\one)$ has the following components (as $\Lambda \one$-modules): 
\begin{gather*}
E^0=\one,\quad 
E^1=0,\quad 
E^2=\one\oplus \one,\quad 
E^3=[\one,\one],\quad
E^4=\one\oplus \one,\\
E^5=[\one,\one],\quad
E^6=\one\oplus\one,\quad 
E^7=0,\quad 
E^8=\one.
\end{gather*} 
The differential maps $E^2=\one\oplus \one\to E^3=[\one,\one]$,
$E^5=[\one,\one]\to E^6=\one\oplus \one$, both by rank $1$ operators, and is zero in other degrees. 
\end{example} 

\subsection{\texorpdfstring
{Ext computations for $p>2$}
{Ext computations for 𝑝 > 2}}

In this section we would like to generalise some of the results of the previous section to the case $p>2$. The constructions and formulas are very similar to the case $p=2$ but not exactly the same due to presence of the invertible object $\psi$ and some other differences, so we chose to repeat them. 

As in the case $p=2$, we can use the resolution $Q_\bullet$ to give the following recursive procedure 
for computation of the additive structure of the cohomology 
$\Ext^\bullet_{\Ver_{p^{n+1}}}(\one,X)$ (for indecomposable $X$). 

For $k\ge 1$ we will denote the generating 
object of $\Ver_{p^{k}}$ by $X_{k-1}$ and recall (\cite{Benson/Etingof/Ostrik}, Subsection 4.14) that 
the principal block $\Ver_{p^{k+1}}^0$ of $\Ver_{p^{k+1}}$ is naturally equivalent 
to the category of $\Lambda X_{k-1}$-modules in 
$\Ver_{p^{k}}$. Let us denote this equivalence by $F$; i.e., 
for an object $X\in \Ver_{p^{k+1}}^0$ we denote the corresponding $\Lambda X_{k-1}$-module by $FX$. 

In the Yoneda realization of ${\rm Ext}$, the Koszul complex $K^\bullet=S^\bullet X_{k-1}\otimes \Lambda X_{k-1}$ represents a class $\tau_k\in\Ext^{p^k-1}_{\Lambda X_{k-1}}(\one,\psi)$, and the class $y_k:=\tau_k^2$ of degree $2(p^k-1)$ is represented by the concatenation 
of $S^\bullet X_{k-1}\otimes \Lambda X_{k-1}$ with $S^\bullet X_{k-1}\otimes \Lambda X_{k-1}\otimes \psi$, which we will denote by $S^\bullet X_{k-1}\otimes \Lambda X_{k-1} \otimes S^\bullet\psi_{k-1}$, where $\psi_{k-1}$ is $\psi$ sitting in degree $p^k-1$. Thus 
$Q_\bullet=S^\bullet X_{k-1}[y_k]\otimes \Lambda X_{k-1}\otimes S^\bullet\psi_{k-1}$.

If $X\in \Ver_{p^n}$ but $X\notin \Ver_{p^n}^0$ then we  have 
$\Ext^\bullet_{\Ver_{p^n}}(\one,X)=0$. So, it suffices to compute 
$\Ext^\bullet_{\Ver_{p^{n+1}}^0}(\one,X)$ for $X\in \Ver_{p^{n+1}}^0$. 
In that case, we have 
\begin{align*}
\Ext^\bullet_{\Ver_{p^{n+1}}^0}(\one,X)&\cong 
\Ext^\bullet_{\Lambda X_{n-1}}(\one,FX)\cong \Ext^\bullet_{\Lambda X_{n-1}}(\one,Q_\bullet\otimes FX)\\
&\cong\Ext^\bullet_{\Lambda X_{n-1}}(\one,S^\bullet X_{n-1}[y_n]
\otimes\Lambda X_{n-1}\otimes S^\bullet \psi_{n-1}\otimes FX) \\
&\cong\Ext^\bullet_{\Ver_{p^n}}(\one,S^\bullet X_{n-1}[y_n]\otimes S^\bullet \psi_{n-1}
\otimes FX) \\
&\cong\Ext^\bullet_{\Ver_{p^n}^0}(\one,(S^\bullet X_{n-1}[y_n]\otimes S^\bullet \psi_{n-1}\otimes FX)^0),
\end{align*}
where the superscript zero means that we are taking the part 
lying in $\Ver_{p^n}^0$, and in the last two expressions 
$FX$ is regarded as an object of $\Ver_{p^n}$ 
using the corresponding forgetful functor 
$\Lambda X_{n-1}$-${\rm mod}\to \Ver_{p^{n}}$ 
forgetting the structure of a $\Lambda X_{n-1}$-module. 

The same calculation applies if $X$ is a complex, i.e., an object 
of the derived category $D^+(\Ver_{p^{n+1}}^0)$ of $\Ver_{p^{n+1}}^0$. 
Namely, for an object $X\in D^+(\Ver_{p^{n+1}}^0)$, let 
\begin{align*} 
E_n(X)&:=\underline{{\rm Hom}}_{\Lambda X_{n-1}}
(\one,S^\bullet X_{n-1}[y_n]\otimes \Lambda X_{n-1}\otimes S^\bullet
  \psi_{n-1}\otimes FX)^0 \\
&=(S^\bullet X_{n-1}[y_n]\otimes S^\bullet\psi_{n-1} \otimes FX)^0. 
\end{align*}
This gives an additive functor $E_n\colon D^+(\Ver_{p^{n+1}}^0)\to 
D^+(\Ver_{p^{n}}^0)$. 

The following lemma is a straightforward analog of Lemma \ref{dif0}.  

\begin{lemma}\label{dif0p}
If $X\in \mathcal \Ver_{p^{n}}$ (with trivial action of $\Lambda X_{n-1}$) then the differential in the complex $E_n(X)$ is zero. 
\end{lemma} 

\begin{corollary}\label{rec1p} 
Suppose $X\in \mathcal \Ver_{p^n}$. Then we have an isomorphism 
\begin{align*}
\Ext^\bullet_{\Ver_{p^{n+1}}^0}(\one,X)&\cong 
\bigoplus_{i\ge 0}\Ext^\bullet_{\Ver_{p^n}}
(\one,S^iX_{n-1}[y_n]\otimes S^\bullet \psi_{n-1}\otimes FX)\\
&=\bigoplus_{i\ge 0}\Ext^\bullet_{\Ver_{p^n}^0}(\one,(S^iX_{n-1}[y_n]
\otimes S^\bullet \psi_{n-1}\otimes FX)^0).
\end{align*}
This isomorphism maps the grading induced by the grading on $\Lambda X_{n-1}$
to the grading defined by $\deg(X_{n-1})=1$, 
$\deg(y_n)=2p^n-2$, $\deg (\psi_{n-1})=p^n-1$ (i.e., it coincides with the cohomological grading).  
\end{corollary} 

As for $p=2$, Corollary \ref{rec1p} does not quite give a recursion to compute the Ext groups, since the object $(S^iX_{n-1}\otimes S^\bullet\psi_{n-1}\otimes FX)^0$ may not belong to $\Ver_{p^{n-1}}$ (i.e., may carry a nontrivial action of $\Lambda X_{n-2}$). However, it has some useful consequences given below. 

\begin{proposition} 
For $n\ge 1$ we have 
\[ \Ext^\bullet_{D^+(\Ver_{p^{n+1}}^0)}(\one,X)=\Ext^\bullet_{D^+(\Ver_{p^n}^0)}(\one,E_n(X)). \]
\end{proposition} 

This implies the following corollary. Let 
$E=E_1\circ\dots\circ E_n\colon \mathcal \Ver_{p^{n+1}}^0\to \Ver_{p}^0={\rm Vec}$. 

\begin{corollary} 
We have a linear isomorphism 
\[ \Ext^\bullet_{D^+(\Ver_{p^{n+1}}^0)}(\one,X)\cong H^\bullet(E(X)). \] 
\end{corollary} 

The complex of vector spaces $E(X)$ has the following structure: 
\begin{multline*} 
E(X)=
(S^\bullet X_0\otimes S^\bullet\psi_0\otimes F(S^\bullet X_1\otimes S^\bullet\psi_1\otimes
\cdots \\ 
\cdots \otimes 
F(S^\bullet X_{n-1}\otimes S^\bullet\psi_{n-1}\otimes
FX)^0\dots)^0)^0[y_1,y_2,\dots,y_n], 
\end{multline*}
and it is easy to see as in the case $p=2$ that the differential is linear over ${\bf k}[y_1,\dots,y_n]$. 
Thus we get
 
\begin{proposition}\label{fingep} 
For any $X\in \Ver_{p^{n+1}}$, $\Ext^\bullet_{\Ver_{p^{n+1}}}(\one,X)$ 
is a graded finitely generated module over ${\bf k}[y_1,\dots,y_n]$. 
\end{proposition}

In particular, we get that 
\[ \Ext^\bullet_{\Ver_{p^{n+1}}}(\one,\one)=
\Ext^\bullet_{\Ver_{p^{n+1}}^0}(\one,\one)=H^\bullet(E(\one)), \]
where 
\begin{multline*}
E(\one)=
(S^\bullet X_0\otimes S^\bullet\psi_0\otimes F(S^\bullet X_1\otimes S^\bullet\psi_1\otimes
\cdots \\ 
\otimes 
F(S^\bullet X_{n-1}\otimes S^\bullet\psi_{n-1})^0\dots)^0)^0[y_1,y_2,\dots,y_n]. 
\end{multline*}
Note that we have $1\in E(\one)$ and $d(1)=0$, so we obtain a 
natural linear map 
\[ \phi\colon {\bf k}[y_1,\dots,y_n]\to 
\Ext^\bullet_{\Ver_{p^{n+1}}}(\one,\one). \] 

\begin{proposition} For $1\le i\le n$ 
multiplication 
by $y_i$ on $\Ext^\bullet_{\Ver_{p^{n+1}}}(\one,X)$ coincides 
with the cup product with $\phi(y_i)$. In particular, $\phi$ is an algebra homomorphism. 
\end{proposition} 

\begin{proof} The proof is the same as that of Proposition \ref{homomo}, using that 
$\phi(y_n)$ can be realised as  Yoneda product 
with the complex $K^\bullet\otimes S^\bullet\psi_{n-1}$, where
$K^\bullet$ is the Koszul complex. The only difference is the presence of the additional factor 
$S^\bullet\psi_{n-1}$. 
\end{proof} 

\begin{proposition}\label{injecp} 
For $X\in \Ver_{p^{n}}^0$ the natural map 
\[ \Ext^\bullet_{\Ver_{p^n}}(\one,X)[y_n]\to
\Ext^\bullet_{\Ver_{p^{n+1}}}(\one,X) \]
is an injective morphism of ${\bf k}[y_1,...,y_n]$-modules 
which is also a morphism of algebras for $X=\one$. 
\end{proposition} 

\begin{proof} 
This follows from the isomorphism 
$$
\Ext^\bullet_{\Ver_{p^{n+1}}}(\one,X)\cong 
\Ext^\bullet_{\Ver_{p^{n}}}(\one,(S^\bullet X_{n-1}
\otimes S^\bullet \psi_{n-1})^0\otimes X)[y_n]
$$
since $\one$ is a direct summand of $(S^\bullet X_{n-1}\otimes S^\bullet \psi_{n-1})^0$. 
\end{proof}

\section{Some further computations}

In order to search for similar patterns for 
 $\Ext^{\bullet}_{\Ver_{p^{n+1}}}(\one,S)$ with $S$ simple in the principal
block, it makes sense to compute a number of examples. For instance, the simplest case 
$n=1$ can be computed using the theory of Brauer tree algebras, and the answer 
is as follows.

Let $X_0,...,X_{N-1}$ label the simple modules for a chain-shaped Brauer tree algebra of length $N$, in the order they occur in the Brauer tree (note that in the case of Verlinde categories $N=p-1$ and $X_0=\one$). Then we have 

\begin{proposition}\label{dudas} (\cite{Dudas}) 
The Poincare series of $\Ext^\bullet(X_i,X_j)$ is given by the formula  
$$
\sum_{k=0}^\infty t^k\dim {\rm Ext}^k(X_i,X_j)=\frac{Q_{ijN}(t)+t^{2N-1}Q_{ijN}(t^{-1})}{1-t^{2N}},
$$
where
$Q_{ijN}(t):=t^{|i-j|} +t^{|i-j|+2} + \cdots + t^{N-1-|N-1-i-j|}$. 
\end{proposition} 

\begin{example} If $i=0$, Proposition \ref{dudas} gives
$$
\sum_{k=0}^\infty t^k\dim {\rm Ext}^k(X_0,X_j)=\frac{t^j+t^{2N-1-j}}{1-t^{2N}}. 
$$
\end{example} 

We also computed $\Ext^{\bullet}_{\Ver_{p^3}}(\one,S)$ for $S$ simple in the
cases $p=2$ and $p=3$. For $p=2$, by the results of
\cite{Benson/Etingof:2019a}, 
we have the following (note that $L_0=\one$):
\[
\sum_{i=0}^\infty t^i \dim\Ext^i_{\Ver_{2^3}}(\one,\one)=
\frac{1+t^2}{(1-t)(1-t^3)},\quad
\sum_{i=0}^\infty t^i \dim\Ext^i_{\Ver_{2^3}}(\one,L_2)=
\frac{t}{1-t^3}. \]

For $p=3$, the Poincar\'e series 
computed using {\sc Magma} agree at least up to degree $100$
with the following (again $L_0=\one$):
\begin{align*}
\sum_{i=0}^\infty t^i \dim\Ext^i_{\Ver_{3^3}}(\one,\one)&=
\frac{1+t^3+t^6+2t^7+t^8+t^{10}+2t^{11}+t^{12}+t^{15}+t^{18}}{(1-t^4)(1-t^{16})}\\
\sum_{i=0}^\infty t^i \dim\Ext^i_{\Ver_{3^3}}(\one,L_4)&=
\frac{t+t^4+t^5+t^6+t^8+2t^9+t^{10}+t^{12}+t^{13}+t^{14}+t^{17}}
{(1-t^4)(1-t^{16})}\\
\sum_{i=0}^\infty t^i \dim\Ext^i_{\Ver_{3^3}}(\one,L_6)&=
\frac{t^2+t^{13}}{1-t^{16}}\\
\sum_{i=0}^\infty t^i \dim\Ext^i_{\Ver_{3^3}}(\one,L_{10})&=
\frac{t^2+2t^3+t^4+t^7+t^8+t^{10}+t^{11}+t^{14}+2t^{15}+t^{16}}
{(1-t^4)(1-t^{16})}\\
\sum_{i=0}^\infty t^i \dim\Ext^i_{\Ver_{3^3}}(\one,L_{12})&=
\frac{t+t^2+t^4+t^5+t^6+2t^9+t^{12}+t^{13}+t^{14}+t^{16}+t^{17}}
{(1-t^4)(1-t^{16})}\\
\sum_{i=0}^\infty t^i \dim\Ext^i_{\Ver_{3^3}}(\one,L_{16})&=
\frac{t^5+t^{10}}{1-t^{16}}.
\end{align*}

\bibliographystyle{amsplain}
\bibliography{../repcoh}

\end{document}